\documentclass[12pt,a4paper]{article}

\usepackage{graphicx,hyperref,enumitem}

\usepackage{amsmath,amsthm,amssymb,mathtools,color}
\usepackage{hyperref}

\usepackage{tikz}
\usetikzlibrary{calc}
\usepackage{subcaption}
\usepackage{caption}
\usepackage{array}
\usepackage{graphicx}
\usepackage{tabularx}

\usepackage[font=small,labelfont=bf]{caption}
\normalsize

\newtheorem{thm}{Theorem}[section]
\newtheorem{cor}[thm]{Corollary}
\newtheorem{lemma}[thm]{Lemma}

\newtheorem{defn}[thm]{Definition}
\newtheorem{claim}[thm]{Claim}

\makeatletter
\def\th@remark{%
  \normalfont 
  \thm@headfont{\bfseries} 
  \thm@headpunct{.}%
  \thm@preskip\topsep
  \thm@postskip\thm@preskip
}
\makeatother

\theoremstyle{remark}
\newtheorem{rmk}{Remark}[section] 

\numberwithin{equation}{section}

\def\Pa{\mathbb{P}}
\newcommand{\db}{\boldsymbol{d}}
\newcommand{\sbb}{\boldsymbol{s}}
\newcommand{\tb}{\boldsymbol{t}}
\newcommand{\lb}{\boldsymbol{l}}
\newcommand{\hb}{\boldsymbol{h}}
\newcommand{\xb}{\boldsymbol{x}}
\newcommand{\yb}{\boldsymbol{y}}
\newcommand{\nb}{\boldsymbol{n}}
\newcommand{\ib}{\boldsymbol{i}}
\newcommand{\mb}{\boldsymbol{m}}
\newcommand{\zb}{\boldsymbol{z}}
\newcommand{\gb}{\boldsymbol{g}}
\newcommand{\zerob}{\mathbf{0}}
\newcommand{\Gs}{\mathcal{G}}
\newcommand{\Bs}{\mathcal{B}}
\newcommand{\N}{\mathbb{N}}
\newcommand{\R}{\mathbb{R}}
\newcommand{\ebar}{\overline{e}}
\newcommand{\emptygraph}{\boldsymbol{\emptyset}}
\newcommand\numberthis{\addtocounter{equation}{1}\tag{\theequation}}

\title{Subgraphs in Random Graphs with Specified Degrees and Forbidden Edges}
\author{
John Larkin\thanks{Mathematical Sciences Institute, Australian National University, Australia. Current Location: Institut f\"ur Informatik, Universit\"at Heidelberg, Germany. Email: \tt larkin@informatik.uni-heidelberg.de} \\
\small  \and 
Brendan D. McKay\thanks{School of Computing,
Australian National University. Email: \tt brendan.mckay@anu.edu.au} \\
\small \and  
Fang Tian \thanks{Department of Applied Mathematics,
Shanghai University of Finance and Economics. Email:\\
\tt tianf@mail.shufe.edu.cn} \\
\small }
\date{}
\begin{document}

\maketitle

\begin{abstract}
Let $G$ be a uniformly chosen simple (labelled) random graph with given degree sequence $\db$ and let $X,Y,L$ be edge-disjoint graphs on the same vertex set as $G$. We investigate the probability that $X \subseteq G$ and that $G \cap Y = \emptygraph$ both conditioned on the event $G \cap L = \emptygraph$. We improve upon known bounds of these probabilities and extend them to a wider range of degree sequences through a more precise edge switching argument.
Notably, a few vertices of linear degree are permitted
provided that the subgraph~$X$ does not have an edge
incident with them.
Further, the graph $L$ is permitted to contain many edges (we provide an example where $L$ is a spanning $r$-regular subgraph with $r = o(n)$).

We provide the same analysis when $G$ is a simple (labelled) bipartite random graph with a given degree sequence $(\sbb,\tb)$. Our work extends the results of Gao and Ohapkin (2023) and McKay (1981, 2010).
\end{abstract}

\section{Introduction}
Let $K_n$ denote the complete graph on the vertex set $W = \{w_1,w_2,\dots,w_n\}$. Graphs are defined with respect to their edge sets, for example $X \subseteq K_n$ denotes a graph on the vertex set $W$ sharing some (potentially all) edges of $K_n$. We call a subgraph of $K_{n}$ a \textit{generic} graph to distinguish it from the bipartite case. For vertices $u,v \in W$, we denote an edge by $uv \in K_n$. Given a graph $G \subseteq K_n$ and a vertex $w \in W$, the notation $d_w$ denotes the degree of vertex $w$ in $G$. Define $\db(G) \coloneq (d_{w_1},\dots,d_{w_n})$ as the degree sequence of $G$ on the vertex set $W$. Denote $\emptygraph\subseteq K_n$ as the graph with $n$ vertices and no edges and define the zero vector as $\zerob \coloneq (0,0,\dots,0)$, where the length will be clear from context.

Given an $n$-tuple of integers $\db = (d_{w_1},d_{w_2},\dots,d_{w_n})$ and graphs $X,Y \subseteq K_n$ such that $X \cap Y = \emptygraph$, define
\begin{equation*}
   \Gs_{\db}(X,Y) =  \{G \subseteq K_n \mid \db(G) = \db \text{, } X \subseteq G \text{ and } G \cap Y = \emptygraph\}
\end{equation*}
as the set of graphs with degree sequence $\db$ containing all the edges of $X$ and not containing any edge in $Y$. Throughout the paper we will assume $\Gs_{\db}(X,Y)$ is non-empty. We endow $\Gs_{\db}(X,Y)$ with the uniform probability measure and write $G \sim \Gs_{\db}(X,Y)$ to denote a random variable distributed according to $\Pa(G = G') = \frac{1}{|\Gs_{\db}(X,Y)|}$ for all $G' \in \Gs_{\db}(X,Y)$. In this case we call $G$ a random graph.

Let $X,Y,L \subseteq K_n$ be edge-disjoint graphs and $G \sim \Gs_{\db}(\emptygraph,L)$. The focus of this paper is to improve upon known bounds on both
\begin{equation*}
    \Pa(X \subseteq G) = \frac{|\Gs_{\db}(X,L)|}{|\Gs_{\db}(\emptygraph,L)|} \qquad \text{and} \qquad \Pa(Y \cap G = \emptygraph) = \frac{|\Gs_{\db}(\emptygraph,L \cup Y)|}{|\Gs_{\db}(\emptygraph,L)|}.
\end{equation*}
To avoid being repetitious, graphs will be denoted by uppercase Roman letters, for example $X_i,Y, L \subseteq K_n$, and their degree sequences $\xb_i,\yb,\lb$ by the corresponding lowercase bold letters $\xb_i = (x_{i,w_1},x_{i,w_2},\dots,x_{i,w_n})$, $\yb = (y_{w_1},\dots,y_{w_n})$ and $\lb = (l_{w_1},\dots,l_{w_n})$. If we want to specify the degree of a vertex $w \in W$, we use the notation $x_{i,w},y_w$ and $l_w$, respectively.

\begin{rmk}
    Let $H \subseteq X \subseteq K_n$ and $L \subseteq K_n$, where $X \cap L = \emptygraph$. For $G \sim \Gs_{\db}(H,L)$, $\Pa(X \subseteq G)$ is the probability that the edges of $X$ are present, given that the edges of $H$ are present and the edges of $L$ are not present. Observe that $|\Gs_{\db}(H,L)| = |\Gs_{\db-\hb}(\emptygraph,L\cup H)|$. Then letting $X' \coloneq X \setminus H$ and taking the random graph $G' \sim \Gs_{\db - \hb}(\emptygraph, L \cup H)$ one has $\Pa(X \subseteq G) = \Pa(X' \subseteq G')$. Therefore, we do not lose the ability to condition on a set of edges being present if we only consider random graphs of the form $G \sim \Gs_{\db}(\emptygraph,M)$ for $M \subseteq K_n$.
\end{rmk}

We repeat the above analysis in the case of random bipartite graphs. Let $K_{m,n}$ denote the complete bipartite graph on the vertex set $S \cup T$, where one part consists of vertices in $S \coloneq \{u_1,\dots,u_m\}$ and the other part contains vertices in $T = \{v_1,\dots,v_n\}$.
For $u \in S$ and $v \in T$, we denote an edge by $uv \in K_{m,n}$ where we always require the first entry (e.g. $u$) to be a vertex in $S$ and the second entry (e.g. $v$) to be a vertex in $T$. Let $\emptygraph \subseteq K_{m,n}$ denote the bipartite graph with no edges. Given a graph $G \subseteq K_{m,n}$ and vertices $u \in S$ and $v \in T$, we reserve the notation $s_u$ to denote the degree of vertex $u$, and $t_v$ to denote the degree of vertex $v$. Given $G \subseteq K_{m,n}$, let $\db(G) \coloneq (\sbb,\tb)$ denote its degree sequence, where $\sbb = (s_{u_1},\dots,s_{u_m})$ and $\tb = (t_{v_1},\dots,t_{v_n})$.

Given a degree sequence $(\sbb,\tb) = (s_{u_1},\dots,s_{u_m}, t_{v_1},\dots,t_{v_n})$ and graphs $X,Y \subseteq K_{m,n}$ such that $X \cap Y = \emptygraph$, define
\begin{equation*}
   \Bs_{(\sbb,\tb)}(X,Y) =  \{G \subseteq K_{m,n} \mid \db(G) = (\sbb,\tb) \text{, } X \subseteq G \text{ and } G \cap Y = \emptygraph\}
\end{equation*}
as the set of bipartite graphs $G$ with degree sequence $(\sbb,\tb)$, where the edges of $X$ are present and the edges of $Y$ are not present. The setup of the probability space is similar, so we omit the details. Write $G \sim \Bs_{(\sbb,\tb)}(X,Y)$ to denote a random variable uniformly taking values in $\Bs_{(\sbb,\tb)}(X,Y)$. Again, to avoid being repetitious, we will denote bipartite graphs by uppercase letters $X,H \subseteq K_{m,n}$ and their respective degree sequences by pairs $(\xb,\yb)$ and $(\hb,\ib)$, where $\xb = (x_{u_1},\dots,x_{u_m}),\hb = (h_{u_1},\dots,h_{u_m})$ are associated to the vertex set $S$\, and $\yb,\ib$ are similarly associated to the vertex set $T$.

\subsection{Outline of Paper}
We introduce more required notation and definitions in Section \ref{notation_and_defn_section} and highlight some previous work. Section \ref{section:degree_sum_function} introduces a simple function as well as some parameters appearing in the results. Section \ref{main_results_section} presents the improved probabilities associated with generic random graphs. We include probabilities of a single edge being present and multiple edges all being present or forbidden. All results can further condition on a set of edges being forbidden. Section \ref{bipartite_random_graphs_section} introduces parallel results in the bipartite case. Sections \ref{generic_case_proofs} and \ref{bipartite_proofs_section} present the edge switching proofs for the generic and bipartite cases, respectively. Section \ref{section:example_calculation} provides some applications of the results.

\section{Notation and Definitions}\label{notation_and_defn_section}

If $G \subseteq K_n$ (resp. $G \subseteq K_{m,n}$) has degree sequence $\db$ (resp. $(\sbb,\tb)$), then write $m(G) \coloneq \frac{1}{2}\sum_{w \in W} d_w$ (resp. $m(G) \coloneq \sum_{u \in S} s_u = \sum_{v\in T} t_v$). We use this notation even when $G$ is a random graph since the degree sequence is always known. 

Let $\gb = (g_{w_1},\ldots,g_{w_k})$ be a list of vertex degrees and $A = \{w_1,\dots,w_k\}$ the corresponding vertex set. $\varDelta(\gb)$ denotes the maximum degree in $\gb$. If $X \subseteq K_{n}$ or $X \subseteq K_{m,n}$, we define $\partial X$ as the set of vertices belonging to at least one edge of $X$. If $A \cap \partial X \neq \emptyset$, then define
\begin{equation*}
    \varDelta_{\partial X}(\gb) \coloneq \max\{g_u \mid u \in A \cap \partial X \}.
\end{equation*}
For example, if $X$ is a perfect matching in $K_{m,n}$, then $\partial X = S \cup T$ and $\varDelta_{\partial X}(\sbb) \coloneq \varDelta(\sbb)$, where $\sbb$ is a degree sequence defined on the vertex set $S$. Finally, $[x]_i \coloneq \prod_{j=0}^{i-1} (x - j)$ denotes the falling factorial with $[x]_0=1$.

We will make use of standard asymptotic notation.
For $\{a_i\}_{i=1}^\infty ,\{b_i\}_{i=1}^\infty$ sequences of real numbers, write $a_i = O(b_i)$ if there exists a constant $C > 0$ such that for all $i$, $|a_i| \leq C\cdot  |b_i|$. Write $a_i= o(b_i)$ if eventually $b_i > 0$ and $\lim_{i \to \infty} \frac{a_i}{b_i}$ exists and equals $0$. Therefore, $a_i = b_i(1+o(1))$ means $\lim_{i\to\infty} \frac{a_i}{b_i} = 1$. Denote $a_i = \Omega(b_i)$ if eventually $a_i > 0$ and $b_i = O(a_i)$. Further, $a_i = \Theta(b_i)$ if both $a_i = O(b_i)$ and $a_i = \Omega(b_i)$. We often apply the above in the context of probabilities of events associated to a sequence of random graphs $(G_k)_{k \in \N}$, where each $G_k \sim \Gs_{\db_k}(X_k,Y_k)$ and $\db_k,X_k,Y_k$ depend on $k$. The behaviour of these parameters in relation to a function $\rho:\N \to (0,1)$ will govern the asymptotic probabilities of these events. To ease notation, we will not always write the subscript $k$ (for example in Corollary \ref{multiple_edge_cor}).

\subsection{Previous Work}
A fruitful approach to estimating $\Pa(X \subseteq G)$ for $G \sim \Gs_{\db}(\emptygraph,L)$ was discovered by McKay in \cite{mckay81} and is commonly referred to as the `method of switchings'. The idea is to define a relation between two non-empty finite sets $A$ and $B$ in order to estimate the ratio $\frac{|A|}{|B|}$. In \cite{mckay81}, a 2-switching (see Figure \ref{2_switch}) was used to estimate the ratio $\frac{|\Gs_{\db}(H + uv,L)|}{|\Gs_{\db}(H,L+uv)|}$. Further calculations then give the probability $\Pa(X \subseteq G)$.

Later, Gao and Ohapkin \cite{gao22} used a  3-switching (see Figure \ref{3_switch}) to bound the same ratio, resulting in the following theorem. Their switching argument relies on the parameter $D(\Delta)$, which is the sum of the $\Delta \coloneq \varDelta(\db)$ largest degrees in $\db$.
\begin{thm}[Gao and Ohapkin, Theorem 4 \cite{gao22}]
    Let $X,L \subseteq K_n$ with $X \cap L = \emptygraph$. Let $G \sim \Gs_{\db}(\emptygraph,L)$.
\begin{description}
    \item[\textbf{Upper Bound:}] If $R_1(\db,X,L) \leq 1$, then
    \begin{equation*}
        \Pa(X \subseteq G) \leq \frac{\prod_{i \in W} [d_i]_{x_i}}{2^{m(X)}[m(G)]_{m(X)}} \left(1 + R_1(\db,X,L) \right)^{m(X)},
    \end{equation*}
    where
    \begin{equation*} 
        R_1(\db,X,L) \coloneq \frac{6D(\Delta)+2\Delta(8+2\Delta(\lb))}{2m(G) - 2m(X)}+ \frac{4m(L)\Delta^2}{(2m(G) - 2m(X))^2}.
    \end{equation*}
    \item[\textbf{Lower Bound:}] If $R_2(\db,X,L) \leq 1$, then
    \begin{equation*}
        \Pa(X \subseteq G) \geq \frac{\prod_{i \in W} [d_i]_{x_i}}{2^{m(X)}[m(G)]_{m(X)}} \left(1 - R_2(\db,X,L) \right)^{m(X)},
    \end{equation*}
    where 
    \begin{equation*} 
        R_2(\db,X,L) \coloneq \frac{2D(\Delta)+6\Delta + 2\Delta(\lb)\Delta + \Delta_{\partial X}(\db)^2}{2m(G) - 2m(X)}.
    \end{equation*}   
\end{description}
\end{thm}
After setting conditions on the error functions $R_1$ and $R_2$, it is possible to obtain the probability to within $1+o(1)$ error. Our Theorem \ref{multiple_edge_thm} differs by improving the range of degree sequences where estimates on $\Pa(X \subseteq G)$ can be calculated and further tightening the error functions. The improvement arises largely from two sources; the first is using a 2-switching for the upper bound of $\Pa(X \subseteq G)$ and a 3-switching for its lower bound. The second is the introduction of a function $k \mapsto D(k)$, where $D(k)$ is the sum of the $k$ largest degrees in $\db$. 

McKay \cite[Theorem 4.6]{mcKay_01_line_sum_symmetric} obtained an asymptotic formula for the number of graphs with degree sequence $\gb$ and no edges in common with $X$. The probability that the edges of $X$ are present is then a ratio of two enumeration results, that is $\frac{|\Gs_{\gb-\xb}(\emptygraph,X)|}{|\Gs_{\gb}(\emptygraph,\emptygraph)|}$. In order to apply this theorem to obtain a probability accurate to a $(1+o(1))$ factor, one at least needs $\varDelta(\gb) = o(m(G)^{\frac{1}{4}})$.
For dense $\gb$, McKay \cite[Theorem 1.3]{subgraphs11mckay} provided enumeration results for $|\Gs_{\gb}(\emptygraph,X)|$ provided that each $g_{w_i}$ does not deviate too much from the average degree. Liebenau and Wormald \cite[Theorem 1.6]{liebenau_wormald_2024} gave a formula for $\Pa(uv \in G)$ when the degree sequence has no degrees deviating too much from the average. Their results cover dense degree sequences which our theorems do not cover, while our results allow some vertex degrees to differ substantially from the average degree.
Isaev and McKay~\cite{mckayisaev} provided a formula for the number
of factors with given degrees in a specified dense graph, allowing
a very wide degree variation.

Results regarding  subgraph probabilities have also been obtained by Kim, Sudakov and Vu \cite[Lemma 2.1]{kim07} to estimate $\Pa(X \subseteq G)$ for a graph $X$ with a fixed number of edges and $G \sim \Gs_{\db}(\emptygraph,\emptygraph)$. Here $\db = (d_{w_i})_{i=1}^n$ and each $d_{w_i} = d(1+o(1))$ for some $d = o(n)$ increasing with $n$. This result was extended by D\'{\i}az, Joos, K\"uhn, and Osthus \cite[Lemmas 2.1,2.2 and Corollary 3.1]{diaz19} to $d$-regular $r$-uniform hypergraphs also allowing one to condition on sets of edges being present or not present. 

\subsection*{Bipartite Case} 
Subgraph probabilities in the bipartite case were also derived by McKay in \cite[Theorem 3.5]{mckay81} and \cite[Theorem 2.2]{mckay2010}. After some calculations one can derive the following result.
\begin{thm}[Theorem 3.5 \cite{mckay81} - Simplification]\label{mckay81_thm_simplification}
    Let $G \sim \Bs_{(\sbb,\tb)}(\emptygraph,\emptygraph)$ and $X \subseteq K_{m,n}$ with degree sequence $(\xb,\yb)$. Let $g_{max} \coloneq \max\{\varDelta(\sbb),\varDelta(\tb)\}$ denote the maximum degree in $(\sbb,\tb)$. If $m(X) \cdot g_{max}^2 = o(m(G) - m(X))$ and $\Bs_{(\sbb,\tb)}(\emptygraph,X) \neq \emptyset$, then
    \begin{equation*}
        \Pa(X \subseteq G) = \frac{\prod_{i \in S} [s_i]_{x_i} \prod_{j \in T} [t_j]_{y_{j}}  }{ [m(G)]_{m(X)}} \cdot (1+o(1)).
    \end{equation*}
\end{thm}
Our theorems increase the range of degree sequences where these probabilities can be calculated. As a simple example, one could replace the above condition by $m(X) \varDelta(\sbb)\varDelta(\tb) = o(m(G)-m(X))$ and obtain the same result, as we show in more generality in Corollary \ref{multiple_edge_cor_bipartite}.

McKay \cite[Theorem 4.6]{mcKay_01_line_sum} found an asymptotic formula for the number of 0-1 matrices with prescribed row and column sums. One can find the probability that a pairing produced using the configuration model avoids all edges in some bipartite graph $X$. The enumeration results then yield a probability that $X$ is present in a random bipartite graph $G \sim \Bs_{(\sbb,\tb)}(\emptygraph,\emptygraph)$, where it is necessary that $g_{max} = o(m(G)^{\frac{1}{4}})$ to obtain a probability to within  a $(1+o(1))$ factor.

Greenhill and McKay \cite[Theorem 2.2]{greenhillmckay09} provided probabilities that a random graph $G \sim \Bs_{(\sbb,\tb)}(\emptygraph,\emptygraph)$ avoids or contains all edges in a subgraph $H$. Their result can be applied when the degrees in $\sbb$ (resp.~$\tb$) do not differ too much from the average degree in $\sbb$ (resp.~$\tb$). This theorem can be applied to dense degree sequences where our theorems do not apply. However, with a highly skewed degree sequence a vertex degree in one component may significantly differ from the average degree in that component, in this case one can apply our Lemma \ref{single_edge_prob_lemma_bipartite}.

Liebenau and Wormald \cite[Theorem 1.4]{liebenau_wormald_bipartite} provided a formula for the probability that a single edge is present in a bipartite graph with a given degree sequence. Again, the degrees in each component cannot differ too much from the average degree in that component. However, their results also apply in the dense cases where our theorems do not.

\section{Degree Sum Function}\label{section:degree_sum_function}
The following definition can be applied to any degree sequence $\db$ defined with respect to some set $W$ of vertices. For example, we will use this definition for bipartite graphs with degree sequence $(\sbb,\tb)$, where we consider $\sbb$ a degree sequence with respect to $S$ and $\tb$ a degree sequence with respect to $T$.

Then without loss of generality, let $\db$ be a degree sequence with respect to the vertex set $W$ and $A \subsetneq W$. The below function $k \mapsto L^{A}(\db,k)$ is the sum of the $k$ largest degrees in $\db$, where the sum \textbf{excludes} degrees of vertices in $A$.

\begin{defn}[Lazy Degree Sum Function]\label{lazy_deg_sum_defn}
Let $\db = (d_{w_i})_{i=1}^n$ be a degree sequence on $W = \{w_1,\dots,w_n\}$ and $A \subsetneq W$. Let $d_{i_1} \geq d_{i_2} \geq  \dots \geq d_{i_{n-|A|}}$ be a non-increasing ordering of the degrees from vertices in $W\setminus A$. For $k \in \{1,2,\dots,n-|A|\}$, the \textit{lazy degree sum function} $L^A(\db,\cdot): \{0,1,\dots,n-|A|\} \to \N_{\geq 0}$ is given by
    $$L^{A}(\db,k) \coloneq \sum_{j=1}^{k} d_{i_j}. $$
\end{defn}
We will only use this notation when $k\le n-|A|$.
As a simple example, for $G \subseteq K_n$ a graph with degree sequence $\db$, $L^{\{w_3\}}(\db(G),n-1) = 2m(G) - d_{w_3}$. 

One idea to improve the probabilistic bounds involves extending the lazy degree sum function in a natural manner to all of $\R$. This is particularly useful since the extended function is non-decreasing and concave.

\begin{defn}[Degree Sum Function]
Consider the setup of Definition \ref{lazy_deg_sum_defn}. The \textit{degree sum function} $D^A(\db,\cdot): \R \to \R$ is the natural piecewise linear extension of the lazy degree sum function. It is given by
    \begin{equation*}
        D^A(\db,x) \coloneq \begin{cases}
            d_{i_1}x, & \text{if } x < 1;  \\
            d_{i_{k+1}}(x-k) +L^{A}(\db,k), & \text{if } x \in [k,k+1) \text{ for some } k \in [n-|A|-1];  \\
            L^{A}(\db,n-|A|), & \text{if } x \geq n-|A|.
        \end{cases}
    \end{equation*}
    If the degree sequence $\db$ is clear from context, as in the case of calculating probabilities of random graphs with degree sequence $\db$, then we omit it from the notation and write $D^A(\db,x) \coloneq D^A(x)$. If $A = \emptyset$, we drop it from the notation. If $H$ is a graph defined on the vertex set $W$ with $\db-\hb$ non-negative, then $D_H^A(x) \coloneq D^A(\db-\hb,x)$.
\end{defn}

The degree sum function enjoys certain properties which we will repeatedly use throughout the proofs in this paper. They are presented in the following lemma without proof.

\begin{lemma}[Properties of the Degree Sum Function]\label{deg_sum_prop}
Let $\db = (d_{w_i})_{i=1}^n$ be a degree sequence with respect to the vertex set $W$ and fix $A \subsetneq W$. The following properties hold:
\begin{enumerate}[label=(\roman*), leftmargin=2.5em]
    \item If $x \in A$ and $k \in \{0,1,\dots,n-|A|-1\}$, then $d_x + D^{A}(k) \leq D^{A \setminus \{x\}}(k+1)$.
    \item The degree sum function $x \mapsto D^A(x)$ is \textit{non-decreasing} and \textit{concave}.
    \item Let $\gb$ be a degree sequence with $\gb - \db$ non-negative. Then for all $x \in \R$, $D^A(\db,x) \leq D^A(\gb,x)$.
    Moreover, this implies that if $H_1 \subseteq H_2$, then $D_{H_2}^A(x) \leq  D_{H_1}^A(x)$.
\end{enumerate}
\end{lemma}

We now introduce a parameter which is an \textbf{average} of the largest degrees in a given degree sequence. It will arise naturally in the 3-switching calculations (Section \ref{3_switch_section}). In the following definition it is not necessary that $W_1 = W_2$.

\begin{defn}[$\alpha$-value] 
Let $\nb, \lb$ be degree sequences with respect to the vertex set $W_1$ and $\gb,\hb$ degree sequences with respect to the vertex set $W_2$. For $w \in W_2$, define
\begin{equation*}
    \alpha_w(\nb,\lb,\gb,\hb) \coloneq \frac{D(\nb + \lb, \max\{g_w-h_w,1\})}{\max\{g_w-h_w,1\}}.
\end{equation*}
The $\alpha$-value behaves like an average of the $\max\{g_w-h_w,1\}$ largest  degrees in $\nb + \lb$. We take a maximum to avoid division by zero.
\end{defn}
The fact that the $\alpha$-value varies with $w \in W_2$ would complicate some theorems, therefore we choose to uniformly bound it by the $\gamma$-value below. This may be much less than the maximum degree in $\nb + \lb$.
\begin{defn}[Uniform Bound on $\alpha$-Value] 
Let $\nb, \lb$ be degree sequences with respect to the vertex set $W_1$ and $\gb,\hb$ degree sequences with respect to the vertex set $W_2$. Let $w_{min} \in W_2$ denote a vertex of minimum degree in $\gb - \hb$. Define
    \begin{equation*}
        \gamma(\nb,\lb,\gb,\hb) \coloneq \alpha_{w_{min}}(\nb,\lb,\gb,\hb).
    \end{equation*}
    Observe $\alpha_w(\nb,\lb,\gb,\hb) \leq \gamma(\nb,\lb,\gb,\hb) \leq \varDelta(\nb + \lb)$ as $w$ ranges over all vertices in $W_2$. This follows as the function $x \mapsto \frac{D(\nb + \lb, \max\{x,1\})}{\max\{x,1\}}$ is monotonically decreasing in $x$.
\end{defn}

\section{Main Results: Generic Random Graphs}\label{main_results_section}
Before introducing the main theorems we define two functions. When appearing in the theorems, one should consider these as `error' terms, which we will impose restrictions upon to make their contribution negligible.

\begin{defn}
    Let $\db$ be a degree sequence and $H,L \subseteq K_n$ with $H \cap L = \emptygraph$. Let $uv \notin H \cup L$ be an edge. Define
    \begin{equation*}
        f(\db,H,L,uv) \coloneq 1 - \frac{D_{H}(d_{u} + l_{u}) + D_{H}(d_{v} + l_{v})}{2m(G) - 2m(H)}
    \end{equation*}
    and
    \begin{equation*}
        g(\db, H, L, uv) \coloneq  1- \frac{D_H(\alpha_{u}(\db,\lb,\db,\hb)+2) + D_H(\alpha_{v}(\db,\lb,\db,\hb)+2)}{2m(G) - 2m(H)}.
    \end{equation*}
\end{defn}

\subsection{Single Edge Theorems}
The lemma below bounds the probability that the edge $uv$ is present conditioned on the event $H$ is present and all of $L$ is not present.
\begin{lemma}[Single Edge Probability]\label{single_edge_prob_lemma}
    Let $G \sim \Gs_{\db}(H,L)$ with $uv \notin H \cup L$.
    \begin{enumerate}[label=\textbf{(\Alph*)}, leftmargin=*]
    \item If $f(\db, H, L, uv) > 0$, then 
    \begin{equation*}
        \Pa(uv \in G) \leq \left(1 + \frac{2m(G) - 2m(H)}{(d_u - h_u)(d_v-h_v)} \cdot  f(\db,H,L,uv) \right)^{-1}. 
    \end{equation*}
    \item If $g(\db,H,L,uv) > 0$, then
    \begin{equation*}
        \Pa(uv \in G) \geq \left(1+\frac{2m(G) - 2m(H)}{(d_u - h_u)(d_v-h_v)} \cdot \frac{1}{g(\db,H,L,uv)}\right)^{-1}.
    \end{equation*}
    \end{enumerate}
\end{lemma}
\begin{rmk}
Let $\varDelta = \varDelta(\db)$. After setting conditions on $f$ and $g$, we can sandwich the probability $\Pa(uv \in G) = \left(1 + \frac{2m(G) - 2m(H)}{(d_u - h_u)(d_v-h_v)} \right)^{-1}\cdot (1+o(1))$. Indeed, Gao and Ohapkin \cite[Corollary 2]{gao22} obtained this probability up to a factor of $1+o(1)$ under two assumptions: (1) $D(\varDelta) + \varDelta \cdot \varDelta(\lb) = o(m(G) - m(H))$ and (2) $m(L)\cdot \varDelta^2 = o((m(G) - m(H))^2)$. Lemma \ref{single_edge_prob_lemma} shows the same asymptotic probability holds under the weaker assumption $D_H(\varDelta + \varDelta(\lb)) = o(m(G) - m(H))$, which is implied solely by (1). 
\end{rmk}

\subsection{Multiple Edge Theorems}
The following theorem and corollary estimates the probability that the subgraph $X$ is present given the edges of $L$ are all not present.
\begin{thm}[Multiple Edges]\label{multiple_edge_thm}
 Let $G \sim \Gs_{\db}(\emptygraph,L)$ and $X \cap L = \emptygraph$. Let $\{e_i\}_{i=1}^{m(X)}$ be an enumeration of the edges of $X$. Define $X_0 \coloneq \emptygraph$ and $X_i \coloneq X_0 + e_1 + e_2 + \dots + e_i$ for $1 \leq i \leq m(X)$.
    \begin{enumerate}[label=\textbf{(\Alph*)}, leftmargin=*]
    \item 
    If $f(\db, X_{i-1}, L, e_i) > 0$ for all $i$, then
    \begin{equation*}
       \hspace{-1cm} \Pa(X \subseteq G) \leq \frac{\prod_{j \in W} [d_j]_{x_{j}}}{2^{m(X)} [m(G)]_{m(X)}} \cdot \prod_{i=1}^{m(X)} \frac{1}{f(\db, X_{i-1}, L, e_i)}.
    \end{equation*}
    \item 
    If $g(\db, X_{i-1}, L, e_i) > 0$ for all $i$, then
    \begin{equation*}
        \Pa(X \subseteq G) \geq \frac{\prod_{j \in W} [d_j]_{x_{j}} }{ 2^{m(X)}[m(G)]_{m(X)}} \cdot \prod_{i=1}^{m(X)} \frac{g(\db, X_{i-1}, L, e_i)}{\phi(\db,X)},
    \end{equation*}
    where $\phi(\db,X) \coloneq 1 + \frac{\frac{1}{m(X)} \cdot \sum_{uv \in X}d_ud_v }{2m(G) - 2m(X)}.$
\end{enumerate}
\end{thm}

\begin{rmk}
    This result is similar to \cite[Theorem 4]{gao22} but with smaller error terms. This allows us to obtain stronger asymptotic results in the next corollary, which is essentially a simplification of these error terms. However, there is a trade-off between simplicity and accuracy. Indeed, there are degree sequences and graphs $X\subseteq K_n$, where the upper bound in Corollary \ref{multiple_edge_cor} fails to give the result up to a $1+o(1)$ factor. However, by using a better term (labelled $\kappa$ in the proof of the corollary) in the calculation one can still obtain the result to a factor of $1+o(1)$. We provide an example of such a calculation in Section \ref{section:example_calculation}.
\end{rmk}

\begin{cor}[Multiple Edges - Simplified]\label{multiple_edge_cor} Let $G \sim \Gs_{\db}(\emptygraph,L)$ with $X \cap L = \emptygraph$. Let $\rho : \N \to (0,1)$ be any function bounded away from $1$. Let
\begin{equation*}
    \Pi(X) \coloneq \frac{\prod_{j \in W} [d_j]_{x_{j}} }{ 2^{m(X)}[m(G)]_{m(X)}} .
\end{equation*}

\begin{enumerate}[label=\textbf{(\Alph*)}, leftmargin=*]
\item If as $k\to\infty$, $D(\varDelta_{\partial X}(\db + \lb)) \leq \rho(k) \cdot  \left( m(G) - m(X) \right)$, then
\begin{equation*}
\Pa(X \subseteq G) \leq \Pi(X)  \cdot
\begin{cases}
 1 + o(1), & \text{if } m(X)\cdot \rho(k) = o(1); \\
O(1), & \text{if } m(X) \cdot \rho(k) = O(1).
\end{cases}
\end{equation*}
\item If as $k \to \infty$, $D(\gamma(\db,\lb,\db,\xb)+2) \leq \rho(k) \cdot  \left(m(G) - m(X)\right)$, then
\begin{equation*}
\Pa(X \subseteq G) \geq \Pi(X) \cdot 
\begin{cases}
 1 + o(1), & \text{if } m(X) \cdot \rho(k) = o(1) \text{ and } \varLambda(k) = o(1);  \\
\Omega(1), & \text{if } m(X) \cdot \rho(k) = O(1) \text{ and }  \varLambda(k) = O(1),
\end{cases}
\end{equation*}
where $\varLambda(k) \coloneq \frac{m(X)\cdot \varDelta_{\partial X}(\db)^2}{m(G) - m(X)}$. 
\end{enumerate}
\end{cor}

\begin{rmk}
    \textbf{Upper Bound Improvement:} 
    Let $\varDelta = \varDelta(\db)$. An asymptotic upper bound on $\Pa(X \subseteq G)$ to a factor of $1+o(1)$ can be obtained from \cite[Theorem 4]{gao22} under the two assumptions: (1) $m(X) \left( D(\varDelta) + \varDelta + \varDelta \cdot \varDelta(\lb) \right) = o(m(G) - m(X))$ and (2) $m(X) m(L)\varDelta^2 = o((m(G)-m(X))^2)$. Corollary \ref{multiple_edge_cor} obtains the same upper bound under a condition implied solely by (1). Consequently, we can take $m(L)$ to be substantially larger. \textbf{Lower Bound Improvement:} A lower bound on $\Pa(X \subseteq G)$ to a factor of $1+o(1)$ can be obtained from \cite[Theorem 4]{gao22} under the assumptions: (1) $m(X) \left(D(\varDelta) + \varDelta + \varDelta \cdot \varDelta(\lb)\right) = o(m(G) - m(X))$ and (2) $m(X) \varDelta_{\partial X}(\db)^2 = o(m(G)-m(X))$. Corollary \ref{multiple_edge_cor} obtains the same lower bound assuming (2) holds along with the weaker assumption $m(X) D(\gamma(\db,\lb,\db,\xb)+2) = o(m(G)-m(X))$, which is implied by (1) as $\gamma(\db,\lb,\db,\xb) \leq \varDelta(\db + \lb)$.
\end{rmk}

\subsection*{Forbidden Edge Theorems}
The following theorem bounds the probability that the edges of $Y$ are forbidden given the edges of $L_0$ are forbidden.
\begin{thm}[Multiple Forbidden Edges]\label{multiple_forbidden_theorem}
    Let $L_0,Y \subseteq K_n$, where $L_0 \cap Y = \emptygraph$. Choose vertices $p_1,\dots,p_{m(Y)}, q_1,\dots,q_{m(Y)} \in W$ such that $\ebar_j = p_jq_j$ and $\{\ebar_j\}_{j=1}^{m(Y)}$ enumerates the edges of $Y$. For $1 \leq j \leq m(Y)$, define $L_j \coloneq L_0 + \ebar_1 + \dots + \ebar_j$. Let $\lb$ and $\yb$ be the degree sequences of $L_0$ and $Y$, respectively. Let $G \sim \Gs_{\db}(\emptygraph,L_0)$.
\begin{enumerate}[label=\textbf{(\Alph*)}, leftmargin=*]
\item  If $g(\db,\emptygraph,L_{j-1},\ebar_j)>0$ for all $j$,  then
    \begin{equation*}
       \hspace{-4ex} \Pa(Y \cap G = \emptygraph) \leq \prod_{j=1}^{m(Y)} \left(1 + \frac{d_{p_j}d_{q_j}}{2m(G)}\right)^{-1} \hspace{0.07cm} \cdot \prod_{j=1}^{m(Y)} \frac{1}{g(\db, \emptygraph, L_{j-1}, \ebar_j)}.
    \end{equation*}
\item If $f(\db,\emptygraph,L_{j-1},\ebar_j) > 0$ for all $j$, then
    \begin{equation*}
       \hspace{-4ex}  \Pa(Y \cap G = \emptygraph) \geq \prod_{j=1}^{m(Y)} \left(1+\frac{d_{p_j} d_{q_j}}{2m(G)}\right)^{-1} \hspace{0.07cm} \cdot  \prod_{j=1}^{m(Y)} f(\db,\emptygraph,L_{j-1},\ebar_j).
    \end{equation*}
\end{enumerate}
\end{thm}
We can impose conditions on the products in Theorem \ref{multiple_forbidden_theorem} to obtain the following corollary.

\begin{cor}[Multiple Forbidden Edges - Simplified]\label{multiple_forbidden_cor}
  Consider the same setup as Theorem \ref{multiple_forbidden_theorem}. Let $\rho : \N \to (0,1)$ be any function bounded away from $1$ and define
  \begin{equation*}
      \Phi(Y) \coloneq \prod_{j=1}^{m(Y)} \left(1+\frac{d_{p_j} d_{q_j}}{2m(G)}\right)^{-1} .
  \end{equation*}
\begin{enumerate}[label=\textbf{(\Alph*)}, leftmargin=*]
\item If as $k\to\infty$, $D(\gamma(\db,\lb+\yb,\db,\zerob)+2) \leq \rho(k) \cdot m(G)$, then
\begin{equation*}
\Pa(Y \cap G = \emptygraph) \leq  \Phi(Y) \cdot \begin{cases}
1 + o(1), & \text{if } m(Y)\cdot \rho(k) = o(1); \\
O(1), & \text{if } m(Y) \cdot \rho(k) = O(1).
\end{cases}
\end{equation*}
\item If as $k \to\infty$, $D(\varDelta_{\partial Y}(\db +\lb + \yb)) \leq \rho(k) \cdot m(G)$, then
\begin{equation*}
\Pa(Y \cap G = \emptygraph) \geq \Phi(Y) \cdot 
\begin{cases}
1 + o(1), & \text{if } m(Y)\cdot \rho(k) = o(1); \\
\Omega(1), & \text{if } m(Y) \cdot \rho(k) = O(1).
\end{cases}
\end{equation*}
\end{enumerate}
\end{cor}

\section{Main Results: Bipartite Random Graphs}\label{bipartite_random_graphs_section}
\begin{defn}[Functions $S$ and $T$]
Let $(\sbb,\tb)$ be a given bipartite degree sequence. Let $H \subseteq K_{m,n}$ be a bipartite graph with degree sequence $(\hb,\ib)$, where $(\sbb,\tb) - (\hb,\ib) \geq 0$. Define
\begin{align*}
    & S_H^A(x) \coloneq D^A(\sbb-\hb,x), && \text{where } A \subsetneq S; \\
    & T_H^B(x) \coloneq D^B(\tb - \ib, x), && \text{where } B \subsetneq T.
\end{align*}
Both functions inherit all properties from Lemma \ref{deg_sum_prop}. If $H = \emptygraph$ and/or $A,B = \emptyset$, we remove it from the notation. For $x \in [m]$, $S(x)$ is the sum of the $x$ largest degrees in $\sbb$. Similarly for $x \in [n]$, $T(x)$ is the sum of the $x$ largest degrees in $\tb$.
\end{defn}

\begin{defn}
    Let $(\sbb,\tb)$ be a bipartite degree sequence and $H,L \subseteq K_{m,n}$ with $H \cap L = \emptygraph$. Let $(\hb,\ib)$ be the degree sequence of $H$ and $(\lb,\mb)$ the degree sequence of $L$. Let $uv \notin H \cup L$ be an edge in $K_{m,n}$. Define 
 \begin{equation*}
        p(\sbb,\tb,H,L,uv) \coloneq 1 - \frac{T_H(s_u + l_u) + S_H(t_v + m_v)}{m(G) - m(H)}
    \end{equation*}
and
\begin{equation*}
    q(\sbb,\tb,H,L,uv) \coloneq 1-\frac{S_H(\alpha_u(\tb,\mb,\sbb,\hb))+T_H(\alpha_v(\sbb,\lb,\tb,\ib))}{m(G) - m(H)}.
\end{equation*}
\end{defn}

\subsection{Single Edge Theorems}
\begin{lemma}[Single Edge Probability - Bipartite Case] \label{single_edge_prob_lemma_bipartite}
    Let $G \sim \Bs_{(\sbb,\tb)}(H,L)$ with\\ $uv \notin H \cup L$. 
    Let $(\hb,\ib)$ be the degree sequence of $H \subseteq K_{m,n}$.
    \begin{enumerate}[label=\textbf{(\Alph*)}, leftmargin=*]
    \item If $p(\sbb,\tb,H,L,uv) > 0$, then
    \begin{equation*}
        \Pa(uv \in G) \leq \left(1 + \frac{m(G) - m(H)}{(s_u - h_u)(t_v-i_v)} \cdot  p(\sbb,\tb,H,L,uv) \right)^{-1}. 
    \end{equation*}
    \item If $q(\sbb,\tb,H,L,uv) > 0$, then
    \begin{equation*}
        \Pa(uv \in G) \geq \left(1+\frac{m(G) - m(H)}{(s_u - h_u)(t_v-i_v)} \cdot \frac{1}{q(\sbb,\tb,H,L,uv)}\right)^{-1}.
    \end{equation*}
    \end{enumerate}
\end{lemma}
\subsection{Multiple Edge Theorems}

\begin{thm}[Multiple Edges - Bipartite Case]\label{multiple_edges_bipartite}
    Let $G \sim \Bs_{(\sbb,\tb)}(\emptygraph,L)$ and $X \cap L = \emptygraph$. Let $\{e_i\}_{i=1}^{m(X)}$ be an enumeration of the edges of $X$. Let $X_0 \coloneq \emptygraph$ and $X_i \coloneq X_0 + e_1 + e_2 + \dots + e_i$ for $1\leq i\leq m(X)$. Let $(\xb,\yb)$ be the degree sequence of $X$.
\begin{enumerate}[label=\textbf{(\Alph*)}, leftmargin=*]
    \item If $p(\sbb,\tb,X_{i-1},L,e_i) > 0$ for all $i$, then
    \begin{equation*}
       \hspace{-1cm} \Pa(X \subseteq G) \leq \frac{\prod_{i \in S} [s_i]_{x_i} \prod_{j \in T} [t_j]_{y_{j}}  }{ [m(G)]_{m(X)}} \cdot \prod_{i=1}^{m(X)} \frac{1}{p(\sbb,\tb, X_{i-1}, L, e_i)}.
    \end{equation*}
    \item If $q(\sbb,\tb,X_{i-1},L,e_i)>0$ for all $i$,  then
   \begin{equation*}
        \Pa(X \subseteq G) \geq \frac{\prod_{i \in S} [s_i]_{x_i} \prod_{j \in T} [t_j]_{y_{j}}}{[m(G)]_{m(X)}} \cdot \prod_{i=1}^{m(X)} \frac{q(\sbb,\tb, X_{i-1}, L, e_i)}{\phi'(\db,X)},
    \end{equation*}
    where $\phi' (\db,X) \coloneq 1 + \frac{\frac{1}{m(X)} \cdot \sum_{uv \in X}s_ut_v }{m(G) - m(X)}.$
\end{enumerate}
\end{thm}

\begin{cor}[Multiple Edges Bipartite - Simplified]\label{multiple_edge_cor_bipartite} Let $G \sim \Bs_{(\sbb,\tb)}(\emptygraph,L)$ with $X \cap L = \emptygraph$. Let $(\xb,\yb)$ be the degree sequence of $X$ and $(\lb,\mb)$ the degree sequence of $L$. Let $\rho : \N \to (0,1)$ be any function bounded away from $1$. Define
\begin{equation*}
    \Pi'(X) \coloneq \frac{\prod_{i \in S} [s_i]_{x_i} \prod_{j \in T} [t_j]_{y_{j}}  }{ [m(G)]_{m(X)}}.
\end{equation*}

\begin{enumerate}[label=\textbf{(\Alph*)}, leftmargin=*]
\item If as $k\to\infty$, $S(\varDelta_{\partial X}(\tb+\mb)) + T(\varDelta_{\partial X}(\sbb+\lb))  \leq \rho(k) \cdot  \left( m(G) - m(X) \right)$, then
\begin{equation*}
\Pa(X \subseteq G) \leq \Pi'(X) \cdot 
\begin{cases}
1 + o(1), & \text{if } m(X)\cdot \rho(k) = o(1); \\
O(1), & \text{if } m(X) \cdot \rho(k) = O(1).
\end{cases}
\end{equation*}
\item If as $k \to \infty$, $S(\gamma(\tb,\mb,\sbb,\xb)) + T(\gamma(\sbb,\lb,\tb,\yb)) \leq \rho(k) \cdot  \left(m(G) - m(X)\right)$, then
\begin{equation*}
\Pa(X \subseteq G) \geq \Pi'(X) \cdot  
\begin{cases}
1 + o(1), & \text{if } m(X) \cdot \rho(k) = o(1) \text{ and } \varLambda'(k) = o(1);  \\
\Omega(1), & \text{if } m(X) \cdot \rho(k) = O(1) \text{ and }  \varLambda'(k) = O(1),
\end{cases}
\end{equation*}
where $\varLambda'(k) \coloneq \frac{m(X)\cdot \varDelta_{\partial X}(\sbb) \varDelta_{\partial X}(\tb)}{m(G) - m(X)}$.
\end{enumerate}
\end{cor}

\begin{rmk}
   We are able to obtain the same probability as McKay in Theorem \ref{mckay81_thm_simplification} under a weaker condition, namely the assumptions of Parts A and B. Indeed, $m(X) \varDelta(\sbb)\varDelta(\tb) = o(m(G) - m(X))$ is enough to obtain the probability to a $(1+o(1))$ error. This follows using $S(\varDelta_{\partial X}(\tb)) + T(\varDelta_{\partial X}(\sbb)) \leq 2 \varDelta(\sbb)\varDelta(\tb)$, $\gamma(\tb,\zerob,\sbb,\xb) \leq \varDelta(\tb)$ and $\gamma(\sbb, \zerob, \tb,\yb) \leq \varDelta(\sbb)$.
\end{rmk}

\subsection{Forbidden Edge Theorems}
The following theorem bounds the probability that the edges of $Y$ are forbidden given that  the edges of $L_0$ are forbidden.
\begin{thm}[Multiple Forbidden Edges - Bipartite Case]\label{multiple_forbidden_theorem_bipartite}
Let $L_0,Y \subseteq K_{m,n}$, where $L_0 \cap Y = \emptygraph$. Let $\{\ebar_{j}\}_{j=1}^{m(Y)}$ be an enumeration of the edges in $Y$, where each $\ebar_j = p_j q_j$ with $p_j \in S$ and $q_j \in T$. For $1 \leq j \leq m(Y)$, define  $L_j \coloneq L_0 + \ebar_1 + \dots + \ebar_j$.  Let $G \sim \Bs_{(\sbb,\tb)}(\emptygraph,L_0)$.
\begin{enumerate}[label=\textbf{(\Alph*)}, leftmargin=*]
\item  If $q(\sbb,\tb,\emptygraph,L_{j-1},\ebar_j)>0$ for all $j$, then
    \begin{equation*}
       \hspace{-4ex} \Pa(Y \cap G = \emptygraph) \leq \prod_{j=1}^{m(Y)} \left(1 + \frac{s_{p_j}t_{q_j}}{m(G)}\right)^{-1} \hspace{0.07cm} \cdot \prod_{j=1}^{m(Y)} \frac{1}{q(\sbb,\tb, \emptygraph, L_{j-1}, \ebar_j)}.
    \end{equation*}
\item If $p(\sbb,\tb,\emptygraph,L_{j-1},\ebar_j) > 0$ for all $j$,  then
    \begin{equation*}
       \hspace{-4ex}  \Pa(Y \cap G = \emptygraph) \geq \prod_{j=1}^{m(Y)} \left(1+\frac{s_{p_j} t_{q_j}}{m(G)}\right)^{-1} \hspace{0.07cm} \cdot  \prod_{j=1}^{m(Y)} p(\sbb,\tb,\emptygraph,L_{j-1},\ebar_j).
    \end{equation*}
\end{enumerate}
\end{thm}

\begin{cor}[Multiple Forbidden Edges Bipartite -  Simplified]\label{multiple_forbidden_cor_bipartite} 
Consider the setup of Theorem \ref{multiple_forbidden_theorem_bipartite}. Let $(\lb,\mb)$ be the degree sequence of $L_0$ and $(\yb,\zb)$ the degree sequence of $Y$. Let $\rho : \N \to (0,1)$ be any function bounded away from $1$. Define
\begin{equation*}
    \Phi'(Y) \coloneq \prod_{j=1}^{m(Y)} \left(1+\frac{s_{p_j} t_{q_j}}{m(G)}\right)^{-1}. 
\end{equation*}
\begin{enumerate}[label=\textbf{(\Alph*)}, leftmargin=*]
\item If as $ k \to\infty$, $ S(\gamma(\tb,\mb+\zb,\sbb,\zerob)) + T(\gamma(\sbb,\lb + \yb, \tb, \zerob))\leq \rho(k) \cdot m(G)$, then
\begin{equation*}
\Pa(Y \cap G) \leq \Phi'(Y)  \cdot
\begin{cases}
 1 + o(1), & \text{if } m(Y)\cdot \rho(k) = o(1); \\
O(1), & \text{if } m(Y) \cdot \rho(k) = O(1).
\end{cases}
\end{equation*}
\item If as $k \to\infty$, $ S(\varDelta_{\partial Y}(\tb + \mb + \zb)) + T(\varDelta_{\partial Y}(\sbb+ \lb + \yb)) \leq \rho(k) \cdot m(G)$, then
\begin{equation*}
\Pa(Y \cap G) \geq \Phi'(Y)  \cdot
\begin{cases}
1 + o(1), & \text{if } m(Y)\cdot \rho(k) = o(1); \\
\Omega(1), & \text{if } m(Y) \cdot \rho(k) = O(1).
\end{cases}
\end{equation*}
\end{enumerate}
\end{cor}

\section{Proof of Lemma \ref{single_edge_prob_lemma} and Theorems \ref{multiple_edge_thm}, \ref{multiple_forbidden_theorem}}\label{generic_case_proofs}
We will use the method of switchings to obtain bounds on the following ratio. For an edge $uv \in K_n$, write $H + uv$ to mean the graph consisting of the edges of $H$ and the edge $uv$. 
\begin{lemma}\label{ratio_bound_lemma}
    Let $H,L \subseteq K_n$ with $H \cap L = \emptygraph$. Let $uv \notin H \cup L$.
    \begin{enumerate}[label=\textbf{(\Alph*)}, leftmargin=*]
    \item   If $f(\db,H,L,uv) >0$, then
    \begin{equation*}
        \frac{|\Gs_{\db}(H+uv,L)|}{|\Gs_{\db}(H,L+uv)|} \leq \frac{(d_u-h_u)(d_v-h_v)}{2m(G) - 2m(H)} \cdot \frac{1}{f(\db,H,L,uv) }.
    \end{equation*}
    \item It always holds that
    \begin{equation*}
         \frac{|\Gs_{\db}(H+uv,L)|}{|\Gs_{\db}(H,L+uv)|}  \geq \frac{(d_u-h_u)(d_v-h_v)}{2m(G) - 2m(H)} \cdot g(\db, H, L, uv).   
    \end{equation*}
\end{enumerate}
The inequality of Part (B) always holds, but is only useful for probability estimates if we assume $g(\db,H,L,uv) > 0$, since this yields valid probabilities in $[0,1]$.
\end{lemma}
The proof of Lemma \ref{ratio_bound_lemma} is left to the end of this section. Firstly, we use it to prove Theorems \ref{multiple_edge_thm} and \ref{multiple_forbidden_theorem}.

\begin{proof}[Proof: Lemma \ref{single_edge_prob_lemma} and Theorem \ref{multiple_edge_thm}]
We have that
\begin{align*}
    \Pa(X \subseteq G) & = \prod_{i=1}^{m(X)} \frac{|\Gs_{\db}(X_{i},L)|}{|\Gs_{\db}(X_{i-1},L)|} \\
    & = \prod_{i=1}^{m(X)} \left(1+  \frac{|\Gs_{\db}(X_{i-1},L+e_i)|}{|\Gs_{\db}(X_{i-1}+e_i,L)|} \right)^{-1}, \numberthis \label{X_prob} 
\end{align*}
where the last equality follows as $\Gs_{\db}(X_{i-1},L)$ can be partitioned into $\Gs_{\db}(X_{i-1}+e_i,L) \cup \Gs_{\db}(X_{i-1},L+e_i)$. To prove Lemma \ref{single_edge_prob_lemma}, we use the same partitioning argument to obtain $\Pa(uv\subseteq G) = \left(1 + \frac{|\Gs_{\db}(H,L+uv)|}{|\Gs_{\db}(H+uv,L)|}\right)^{-1}$, and then apply Lemma \ref{ratio_bound_lemma}. 

Now we proceed to prove Theorem \ref{multiple_edge_thm}. Let $u_i$ and $v_i$ be vertices such that $e_i = u_iv_i$. Recall that $x_{i-1,u_i}$ is the degree of vertex $u_i$ in the graph $X_{i-1}$.

\textbf{Part A: Upper Bound.}
The bound in Lemma \ref{ratio_bound_lemma}(A) applied to \eqref{X_prob} yields
\begin{align*}
    \Pa(X \subseteq G) & \leq \prod_{i=1}^{m(X)} \left(1 + \frac{2m(G) - 2m(X_{i-1})}{(d_{u_i}-x_{i-1,u_i})(d_{v_i}-x_{i-1,v_i})} \cdot f(\db, X_{i-1}, L, e_i)\right)^{-1} .
\end{align*}
For $a_1,a_2,\dots,a_k \geq 0$, we have $\prod_{i=1}^k \frac{1}{1+a_i} \leq \frac{1}{1+\prod_{i=1}^k a_i}$. Further, observe that $\prod_{i=1}^{m(X)} (2m(G) - 2m(X_{i-1})) = 2^{m(X)} [m(G)]_{m(X)}$ and $\prod_{i=1}^{m(X)} (d_{u_i} - x_{i-1, u_i})(d_{v_i} - x_{i-1, v_i}) = \prod_{j \in W} [d_j]_{x_{j}}$.  This information gives
\begin{align*}
    \Pa(X \subseteq G) & \leq  \left( 1 + \frac{2^{m(X)} [m(G)]_{m(X)}}{\prod_{j \in W} [d_j]_{x_{j}}} \cdot \prod_{i=1}^{m(X)}  f(\db, X_{i-1}, L, e_i) \right)^{-1}. 
\end{align*}
The inequality $\frac{1}{1+\frac{1}{x}} \leq x$ for $x > 0$ concludes the proof.

\textbf{Part B: Lower Bound.}
The bound in Lemma \ref{ratio_bound_lemma}(B) applied to Equation \eqref{X_prob} yields
\begin{align*}
    \Pa(X \subseteq G) & \geq \prod_{i=1}^{m(X)} 
 \left(1 + \frac{2m(G) - 2m(X_{i-1})}{(d_{u_i}-x_{i-1,u_i})(d_{v_i}-x_{i-1,v_i}) \cdot g(\db, X_{i-1}, L, e_i)} \right)^{-1} \\
    & \geq \prod_{i=1}^{m(X)} \frac{(d_{u_i}-x_{i-1,u_i})(d_{v_i}-x_{i-1,v_i}) \cdot g(\db, X_{i-1}, L, e_i)}{(d_{u_i}-x_{i-1,u_i})(d_{v_i}-x_{i-1,v_i}) + 2m(G) - 2m(X_{i-1})} \numberthis \label{prob_lower_bound}.
\end{align*}
For the last inequality, we use $g(\db,X_{i-1},L,e_i) \leq 1$. The numerator is given by $\prod_{j \in W} [d_j]_{x_{j}} \cdot \prod_{i=1}^{m(X)} g(\db, X_{i-1}, L, e_i)$. Consider the term in the denominator. By relaxing $d_w - x_{i-1,w} \leq 
 d_w$ and $m(X_{i-1}) \leq m(X)$, we obtain
{\allowdisplaybreaks
\begin{align*}
    & \prod_{i=1}^{m(X)} \Bigl((d_{u_i}-x_{i-1,u_i})(d_{v_i}-x_{i-1,v_i}) + 2m(G) - 2m(X_{i-1})\Bigr) \\
    & \leq 2^{m(X)}[m(G)]_{m(X)}  \prod_{i=1}^{m(X)} \left(\frac{d_{u_i} d_{v_i}}{2m(G) - 2m(X)} + 1 \right) \\
    & \leq 2^{m(X)}[m(G)]_{m(X)} \cdot \phi(\db,X)^{m(X)}.  
\end{align*}}
The last inequality follows by the AM-GM inequality. This information in Equation \eqref{prob_lower_bound} gives the result.
\end{proof}

\begin{proof}[Proof of Theorem \ref{multiple_forbidden_theorem}]
    We have that
\begin{align*}
    \Pa(Y \cap G = \emptygraph ) & = \prod_{j=1}^{m(Y)} \frac{|\Gs_{\db}(\emptygraph,L_j)|}{|\Gs_{\db}(\emptygraph,L_{j-1})|}  \\
    & = \prod_{j=1}^{m(Y)} \left(1 + \frac{|\Gs_{\db}(\ebar_j,L_{j-1})|}{|\Gs_{\db}(\emptygraph,L_{j-1}+\ebar_j)|}\right)^{-1} \numberthis \label{Y_prob}, 
\end{align*}
where the last equality follows as $\Gs_{\db}(\emptygraph,L_{j-1})$ can be partitioned into $\Gs_{\db}(\emptygraph,L_{j-1}+\ebar_j) \cup \Gs_{\db}(\ebar_j,L_{j-1})$. Both Part A and Part B immediately follow from Lemma \ref{ratio_bound_lemma} applied to \eqref{Y_prob} along with the observation $g(\db,,\emptygraph,L_{j-1},\ebar_j) \leq 1$ and $f(\db,\emptygraph,L_{j-1},\ebar_j) \leq 1$.
\end{proof}

\subsection{Proof of Lemma \ref{ratio_bound_lemma}}
\begin{proof}[Proof of Lemma \ref{ratio_bound_lemma}]
We will use the method of switchings to estimate the ratio $\frac{|\Gs_{\db}(uv,M)|}{|\Gs_{\db}(\emptygraph,M+uv)|} $, where $M \subseteq K_n$ and $uv \notin M$. This suffices as one can make the substitutions $\db \mapsto \db - \hb$ and $M \mapsto H \cup L$ to obtain the estimate on $\frac{|\Gs_{\db-\hb}(uv,H \cup L)|}{|\Gs_{\db-\hb}(\emptygraph,H \cup L + uv)|}  = \frac{|\Gs_{\db}(uv+H,L)|}{|\Gs_{\db}(H,L + uv)|} $ presented in Lemma~\ref{ratio_bound_lemma}.

Our calculations deviate from previous results since we use a different switching for the upper and lower bounds.
Indeed, for the upper bound in Lemma \ref{ratio_bound_lemma}, we adopt a 2-switching that was used 
for both upper and lower bounds in \cite{mckay81}, while  for the lower bound we will use a 3-switching introduced in \cite{enumeration91} and
also used by \cite{gao22}. In the diagrams, solid black lines are edges in $G$ and dashed black lines are non-edges. Red edges are known to be in $G \cup M$.

\subsection*{Part A Proof: 2-Switchings}
\subsubsection*{Forward 2-Switchings}
\begin{defn}[Forward 2-Switching]
    Given $G \in \Gs_{\db}(uv,M)$, we call an ordered $2$-tuple of vertices $(x,y) \in W^2$ a \textit{forward 2-switching} on $G$ if the following conditions hold:
    \begin{enumerate}
        \item $xy$ is an edge in $G$ and $xy \neq uv$.
        \item $\{x,y\} \cap \{u,v\} = \emptyset$ and $ux,vy$ are not edges in $G$ or $M$.
    \end{enumerate}
\end{defn}
\begin{figure}[htbp]
\centering
   \begin{tikzpicture}[
every edge/.style = {draw=black,very thick},
 vrtx/.style args = {#1/#2}{%
      circle, draw, thick, fill=white,
      minimum size=4mm, label=#1:#2}
                    ]
\node(u) [vrtx= left/$u$] at (0-5, 2) {};
\node(v) [vrtx= left/$v$] at (0-5, 0) {};
\node(y) [vrtx= right/$y$] at (2-5, 0) {};
\node(x) [vrtx= right/$x$] at (2-5, 2) {};

\path   (u) edge (v)
        (x) edge (y)
        (u) edge[dashed]   (x)
        (v) edge[dashed]   (y);

\node(u') [vrtx= left/$u$] at (0+3, 2) {};
\node(v') [vrtx= left/$v$] at (0+3, 0) {};
\node(y') [vrtx= right/$y$] at (2+3, 0) {};
\node(x') [vrtx= right/$x$] at (2+3, 2) {};

\path   (u') edge[dashed] (v')
        (x') edge[dashed] (y')
        (u') edge  (x')
        (v') edge  (y');

\node(forward) at (-0,0.3) (<name>) {Forward 2 - Switching};
\draw[->, above, thick] (-1.5,0) -- (1.5,0);

\node(back) at (0,0.3+2) (<name>) {Backward 2 - Switching};
\draw[<-, above, thick] (-1.5,0+2) -- (1.5,0+2);

\end{tikzpicture} 
\caption{A 2-switching between $\Gs_{\db}(uv,M)$ and $\Gs_{\db}(\emptygraph,M+uv)$.}
\label{2_switch}
\end{figure} 

\def\remove#1{}
With reference to Figure \ref{2_switch}, given a forward 2-switching $(x,y)$, we can delete the edges $uv,xy$ and add the edges $ux$ and $vy$ to obtain a new graph $G' \in \Gs_{\db}(\emptygraph, M+uv)$. For this reason, we say $(x,y)$ is a forward 2-switching \textit{from} $G$ to $G'$.

Condition 1 ensures $x \neq y$ and $xy \in G$ is not the edge $uv$. Condition 2 (with Condition 1) ensures $u,v,x,y$ are distinct. Further it imposes that $ux$ and $vy$ are not edges in $G$ and are not forbidden by $M$, since we add these edges during the switching. After performing a switching the degree sequence has not changed. Let $f_2(G)$ denote the number of forward 2-switchings from $G$. We have the following uniform lower bound on $f_2(G)$.
\begin{claim}\label{forward_2_switch_lower}
For all $G \in \Gs_{\db}(uv,M)$,
    \begin{equation*}
        f_2(G) \geq 2m(G) \cdot f(\db,\emptygraph,M,uv). 
    \end{equation*}
\end{claim}
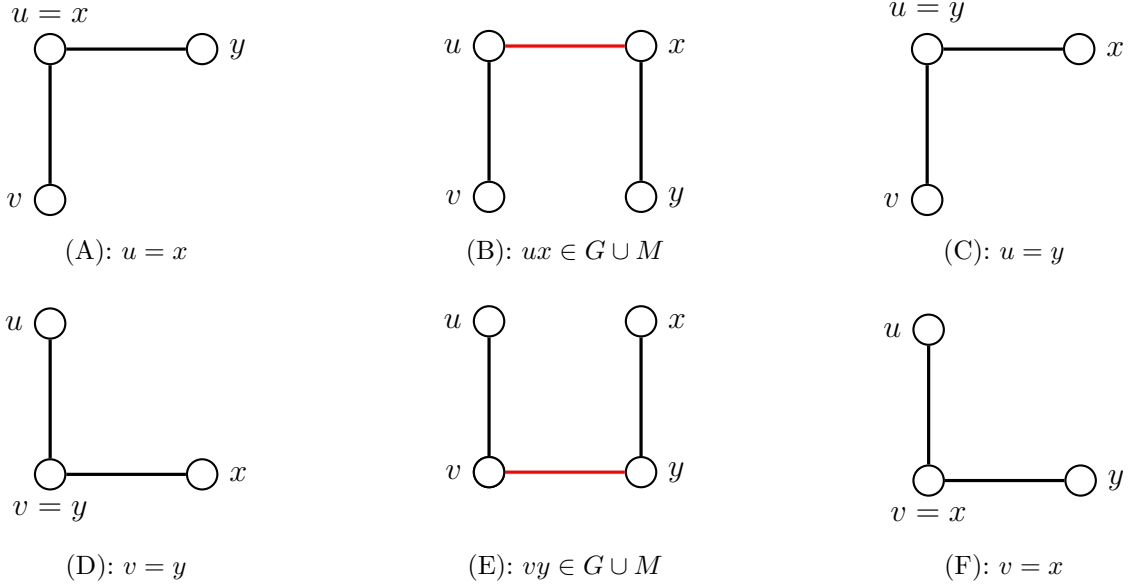
\begin{figure}[ht]
\centering

\begin{subfigure}[b]{0.3\textwidth}
\centering
\begin{tikzpicture}[
  every edge/.style = {draw=black,very thick},
  vrtx/.style args = {#1/#2}{circle, draw, thick, fill=white, minimum size=4mm, label=#1:#2}
]
\node(u) [vrtx=above/\phantom{h}, label = above : {$u=x$}] at (0, 2) {};
\node(v) [vrtx=left/$v$] at (0, 0) {};
\node(x) [vrtx=right/$y$] at (2, 2) {};
\path (u) edge (v);
\path (u) edge (x);
\end{tikzpicture}
\caption*{(A): $u = x$}
\end{subfigure}
\hfill
\begin{subfigure}[b]{0.3\textwidth}
\centering
\begin{tikzpicture}[
  every edge/.style = {draw=black,very thick},
  vrtx/.style args = {#1/#2}{circle, draw, thick, fill=white, minimum size=4mm, label=#1:#2}
]
\node(u) [vrtx=left/$u$] at (0, 2){};
\node(v) [vrtx=left/{$v$}] at (0, 0){};
\node(y) [vrtx=right/{$y$}] at (2, 0){};
\node(x) [vrtx=right/{$x$}] at (2, 2){};
\path (u) edge (v);
\path (x) edge (y);
\path (u) edge[red] (x);
\end{tikzpicture}
\caption*{(B): $ux \in G \cup M$}
\end{subfigure}
\hfill
\begin{subfigure}[b]{0.3\textwidth}
\centering
\begin{tikzpicture}[
  every edge/.style = {draw=black,very thick},
  vrtx/.style args = {#1/#2}{circle, draw, thick, fill=white, minimum size=4mm, label=#1:#2}
]
\node(u) [vrtx=above/\phantom{h}, label = above : {$u=y$}] at (0, 2) {};
\node(v) [vrtx=left/$v$] at (0, 0) {};
\node(x) [vrtx=right/$x$] at (2, 2) {};
\path (u) edge (v);
\path (u) edge (x);
\end{tikzpicture}
\caption*{(C): $u = y$}
\end{subfigure}

\vspace{1em}

\begin{subfigure}[b]{0.3\textwidth}
\centering
\begin{tikzpicture}[
  every edge/.style = {draw=black,very thick},
  vrtx/.style args = {#1/#2}{circle, draw, thick, fill=white, minimum size=4mm, label=#1:#2}
]
\node(u) [vrtx=left/$u$] at (0, 2) {};
\node(v) [vrtx=left/\phantom{h}, label = below : {$v=y$}] at (0, 0) {};
\node(y) [vrtx=right/$x$] at (2, 0) {};
\path (u) edge (v);
\path (v) edge (y);
\end{tikzpicture}
\caption*{(D): $v = y$}
\end{subfigure}
\hfill
\begin{subfigure}[b]{0.3\textwidth}
\centering
\begin{tikzpicture}[
  every edge/.style = {draw=black,very thick},
  vrtx/.style args = {#1/#2}{circle, draw, thick, fill=white, minimum size=4mm, label=#1:#2}
]
\node(u) [vrtx=left/$u$] at (0, 2) {};
\node(v) [vrtx=left/$v$] at (0, 0) {};
\node(y) [vrtx=right/$y$] at (2, 0) {};
\node(x) [vrtx=right/$x$] at (2, 2) {};
\node(fill) [vrtx=below/\phantom{h}, label = below : {\phantom{$v=y$}}] at (0, 0) {};
\path (u) edge (v);
\path (x) edge (y);
\path (v) edge[red] (y);
\end{tikzpicture}
\caption*{(E): $vy \in G \cup M$}
\end{subfigure}
\hfill
\begin{subfigure}[b]{0.3\textwidth}
\centering
\begin{tikzpicture}[
  every edge/.style = {draw=black,very thick},
  vrtx/.style args = {#1/#2}{circle, draw, thick, fill=white, minimum size=4mm, label=#1:#2}
]
\node(u) [vrtx=left/$u$] at (0, 2) {};
\node(v) [vrtx=left/\phantom{h}, label = below : {$v=x$}] at (0, 0) {};
\node(y) [vrtx=right/$y$] at (2, 0) {};
\path (u) edge (v);
\path (v) edge (y);
\end{tikzpicture}
\caption*{(F): $v = x$}
\end{subfigure}

\caption{Bad cases where Condition 2 is not satisfied}
\label{bad_con_2_2_switch}
\end{figure}

\begin{proof}[Proof of Claim \ref{forward_2_switch_lower}]
    We use inclusion-exclusion. There are exactly $2m(G) - 2$ ordered pairs $(x,y)$ such that Condition 1 is satisfied. A $2$-tuple is called \textit{bad} if it satisfies Condition 1 but \textbf{not} Condition 2. Starting with the exact count on $2$-tuples satisfying Condition 1, we remove an upper bound on the number of bad $(x,y)$ to obtain a lower bound for $f_2(G)$.

\textbf{Bad Cases: Upper Bound. }
The bad cases satisfying Condition 1 but not Condition~2 are shown in Figure \ref{bad_con_2_2_switch}. Cases (A), (B) and (C) correspond to $u = x$, $ux \in G \cup M$ and $u = y$. Cases (D), (E) and (F) correspond to $v = y$, $vy \in G \cup M$ and $v = x$. The number of bad $(x,y)$ satisfying (A) is $d_u - 1$. Since we set $u = x$ and have $d_u - 1$ choices for $y$, where $y \neq v$. The number of bad $(x,y)$ satisfying (B) and (C) is
\begin{equation*}
    \sum_{\substack{x \in W \\ ux \in G \cup M \\ x \neq v}}d_x.
\end{equation*}
This arises since for each choice of $x$ such that $ux \in G \cup M$ and $x \neq v$, there are $d_x$ choices for $y$ that result in case (B) or (C). Case (C) only occurs for $x$ such that $ux \in G$ and is counted in the $d_x$ term. Therefore, the number of bad (A), (B) and (C) cases is 
\begin{equation*}
    (d_u - 1) + \sum_{\substack{x \in W \\ ux \in G \cup M \\ x \neq v}} d_x \leq (d_u - 1) + D^{\{u\}}(d_u + m_u - 1).
\end{equation*}
This follows since there are $d_u + m_u - 1$ terms in the summation (the $-1$ as $x \neq v$) and $x \neq u$ since $ux \in G \cup M$. Using $d_u  + D^{\{u\}}(d_u + m_u - 1) \leq D(d_u + m_u)$ gives the number of bad $(x,y)$ in cases (A), (B) and (C) being at most $D(d_u + m_u) - 1$. A similar count shows that the number of bad $(x,y)$ in cases (D), (E) and (F) is at most $D(d_v + m_v) - 1$. Subtracting from the upper bound then gives
\begin{align*}
   f_2(G) &  \geq 2m(G) - 2 - (D(d_u + m_u) + D(d_v + m_v) - 2) \\
    & = 2m(G) \cdot f(\db,\emptygraph,M,uv).\qedhere
\end{align*}
\end{proof}

\subsubsection*{Backward 2-Switchings}
We now define a backward 2-switching.
\begin{defn}[Backward 2-Switching]
    Please refer to Figure \ref{2_switch}. Given $G' \in \Gs_{\db}(\emptygraph,M+uv)$, we call an ordered $2$-tuple $(x,y) \in W^2$ a \textit{backward 2-switching} on $G$ if the following conditions hold:
    \begin{enumerate}
        \setcounter{enumi}{2}
        \item $ux,vy$ are edges in $G'$.
        \item $x \neq y$ and $xy$ is not an edge in $G'$ or $M$.
    \end{enumerate}
\end{defn}

Again, given such a $2$-tuple, we can modify $G'$ by deleting the edges $ux$ and $vy$ and adding the edges $uv$ and $xy$ to obtain $G \in \Gs_{\db}(uv,M)$. Condition 4 ensures $x$ and $y$ are distinct and $xy$ is not an edge in $G'$ or $M$, since we add it to the graph during the switching. Let $b_2(G')$ denote the number of backward 2-switchings from $G'$. We have the following uniform bound on $b_2(G')$.
\begin{claim}\label{backward_2_switch_upper}
    For all $G' \in \Gs_{\db}(\emptygraph,M+uv)$,
    \begin{equation*}
        b_2(G') \leq d_u d_v.
    \end{equation*}
\end{claim}
\begin{proof}
    We count the number of $(x,y)$ satisfying Condition 3. There are $d_u$ choices for $x$ and for each of these there are $d_v$ choices for $y$.
\end{proof}

Let $T$ denote the number of pairs $(G,G')$, where $G \in \Gs_{\db}(uv,M)$ and $G' \in \Gs_{\db}(\emptygraph,M+uv)$ and there exists a forward switching from $G$ to $G'$. Observe that this implies there is a backward switching from $G'$ to $G$, so we obtain the equality
\begin{equation*}
 \sum_{G \in \Gs_{\db}(uv,M)} f_2(G) = T = \sum_{G' \in \Gs_{\db}(\emptygraph,M+uv)} b_2(G'). 
\end{equation*}
Claim \ref{forward_2_switch_lower} and Claim \ref{backward_2_switch_upper} then give
\begin{equation*}
    |\Gs_{\db}(uv,M)| \cdot 2m(G) \cdot f(\db,\emptygraph,M,uv)  \leq |\Gs_{\db}(\emptygraph,M+uv)| \cdot d_u d_v . 
\end{equation*}
The assumption $f(\db,\emptygraph,M,uv) > 0$ allows one to rearrange the above to an upper bound on $\frac{|\Gs_{\db}(uv,M)| }{|\Gs_{\db}(\emptygraph,M+uv)| }$. Observe that the substitution $\db \mapsto \db - \hb$ and $M \mapsto H \cup L$ gives the result as shown in Lemma \ref{ratio_bound_lemma}(A). Indeed the substitution yields, $2m(G) \mapsto 2m(G) - 2m(H)$ and $f(\db,\emptygraph,M,uv) \mapsto f(\db,H,L,uv) $.

\subsection*{Part B Proof: 3-Switchings}\label{3_switch_section}
\begin{figure}[htbp]
\centering
   \begin{tikzpicture}[
every edge/.style = {draw=black,very thick},
 vrtx/.style args = {#1/#2}{%
      circle, draw, thick, fill=white,
      minimum size=4mm, label=#1:#2}
                    ]
\node(u) [vrtx= left/$u$] at (-8*0.75, 2*0.75) {};
\node(v) [vrtx= left/$v$] at (-8*0.75, 0) {};
\node(w) [vrtx= above/$y$] at (-6*0.75, -1.2*0.75) {};
\node(c) [vrtx= above/$x$] at (-6*0.75, 3.2*0.75) {};
\node(a) [vrtx= right/$a$] at (-4*0.75,2*0.75) {};
\node(b) [vrtx= right/$b$] at (-4*0.75,0) {};

\path   (u) edge (v)
        (u) edge[dashed]   (c)
        (c) edge (a)
        (b) edge (w)
        (v) edge[dashed]   (w)
        (a) edge[dashed]   (b);

\node(u') [vrtx= left/$u$] at (2*0.75, 2*0.75) {};
\node(v') [vrtx= left/$v$] at (2*0.75, 0) {};
\node(w') [vrtx= above/$y$] at (4*0.75, -1.2*0.75) {};
\node(c') [vrtx= above/$x$] at (4*0.75, 3.2*0.75) {};
\node(a') [vrtx= right/$a$] at (6*0.75,2*0.75) {};
\node(b') [vrtx= right/$b$] at (6*0.75,0) {};

\path   (u') edge[dashed]  (v')
        (u') edge  (c')
        (c') edge[dashed]  (a')
        (b') edge[dashed]  (w')
        (v') edge  (w')
        (a') edge  (b');

\node(forward) at (-1*0.75,-1+0.3) (<name>) {Forward 3-Switching};
\draw[->, above, thick] (-2.5+0.25,-1) -- (0.5+0.25,-1);

\node(back) at (-1*0.75,2.3) (<name>) {Backward 3-Switching};
\draw[<-, above, thick] (-2.5+0.25,2) -- (0.5+0.25,2);
\end{tikzpicture} 
\caption{A 3-switching between $\Gs_{\db}(uv,M)$ and $\Gs_{\db}(\emptygraph,M+uv)$}
\label{3_switch}
\end{figure} 
\subsubsection*{Forward 3-Switchings}
\begin{defn}[Forward 3-Switching]
    Given $G \in \Gs_{\db}(uv,M)$, we call an ordered $4$-tuple $(x,a,y,b) \in W^4$ a \textit{forward 3-switching} on $G$ if the following conditions hold:
\begin{enumerate}
    \item $uv,xa,yb$ are edges in $G$ (not necessarily distinct).
    \item The vertices $u,v,x,y,a,b$ are distinct except $x = y$ is permitted.
    \item $ux,vy,ab$ are not edges in $G$ or $M$.
\end{enumerate}
\end{defn}

Given such an ordered $4$-tuple, one can modify the graph $G$ into a new graph $G' \in \Gs_{\db}(\emptygraph, M + uv)$ by deleting the edges $uv,xa,yb$ and adding the edges $ux,vy,ab$ as shown in Figure \ref{3_switch}. Condition 2 imposes that $uv,ux,xa,ab,yb,vy$ are distinct edges or non-edges in $G$. Condition 3 ensures $ux,vy,ab$ are not edges in $G$ or $M$, since we add them during the switching. We say $(x,a,y,b)$ is a forward 3-switching from $G$ to $G'$. Let $f_3(G)$ denote the number of forward 3-switchings for a specific $G \in \Gs_{\db}(uv, M)$. We obtain the following bound on $f_3(G)$.

\begin{claim}\label{forward_3_switch_upper}
For all $G \in \Gs_{\db}(uv,M)$,
\begin{equation*} 
   f_3(G) \leq (2m(G))^2.
\end{equation*}
\end{claim}
\begin{proof}
    We exactly count the number of $4$-tuples satisfying Condition 1. There are $2m(G)$ choices for $x$ and $a$, where $xa \in G$; indeed there are $m(G)$ choices for an edge in $G$, and then we choose one endpoint of this edge to be $x$ and the other endpoint to be $a$. Similarly, there are $2m(G)$ choices for $y$ and $b$.
\end{proof}

\subsubsection*{Backward 3-Switchings}
We now define a backward 3-switching.
\begin{defn}[Backward 3-Switching]
    Given $G' \in \Gs_{\db}(\emptygraph, M + uv)$, we call an ordered $4$-tuple $(x,a,y,b) \in W^4$ a \textit{backward 3-switching} on $G'$ if the following conditions hold:
    \begin{enumerate}
    \setcounter{enumi}{3}
    \item $ux,vy,ab$ are edges in $G'$ (not necessarily distinct).
    \item $\{a,b\} \cap \{u,v\} = \emptyset$.
    \item $x \notin \{a,b\}$ and $xa$ is not an edge in $G'$ or $M$.
    \item $y \notin \{a,b\}$ and $yb$ is not an edge in $G'$ or $M$.
\end{enumerate}
\end{defn}

Given such a $4$-tuple, one can modify the graph $G'$ by deleting the edges $ux,vy,ab$ and adding the edges $uv,xa,yb$ to obtain a graph $G \in \Gs_{\db}(uv, M)$. Condition 4 ensures that $ux,vy,ab$ are edges in $G'$ that we delete during the switching. Condition 5 ensures $u$ and $v$ are distinct from $a$ and $b$, which implies the edges $ux,vy$ and $ab$ are distinct when combined with Condition 4. Conditions 6 and 7 imply $xa$ and $yb$ are non-edges in $G'$ and also are not forbidden (edges in $M$). Conditions 4, 5, 6 and 7 together imply $u,v,x,y,a,b$ are distinct except $x=y$ is permitted. Given $G' \in \Gs_{\db}(\emptygraph, M + uv)$, let $b_3(G')$ denote the number of backward switchings from $G'$. We introduce the bounds on $b_3(G')$.
\begin{claim}\label{backward_switch_3_lower_general}
    For all $G' \in \Gs_{\db}(\emptygraph,M+uv)$,
    \begin{equation*}
        b_3(G') \leq 2m(G) d_ud_v
    \end{equation*}
    and
    \begin{equation*}
        b_3(G') \geq 2m(G) d_u d_v \cdot g(\db,\emptygraph,M,uv).
    \end{equation*}
\end{claim}
\begin{proof}
    The upper bound arises by counting $4$-tuples $(x,a,y,b) \in W^4$ satisfying Condition 4. Indeed there are $d_u$ choices for $x$ and $d_v$ choices for $y$ such that $ux,vy \in G'$. There are at most $2m(G)$ choices for the edge $ab \in G'$. Therefore, $b_3(G') \leq 2m(G) d_u d_v$.

For the lower bound we use inclusion-exclusion. Starting from the number of $4$-tuples satisfying Condition 4, we will subtract upper bounds on the number of $4$-tuples satisfying Condition 4 but \textbf{not} satisfying Condition $5,6$ or $7$. We call a $4$-tuple \textit{terrible} if it satisfies Condition 4 but not Condition 5. We call a $4$-tuple \textit{bad} if it satisfies Conditions 4 and 5 but not Condition $6$ or $7$. We will first provide an exact count of the terrible $4$-tuples and then an upper bound on the number of bad $4$-tuples. The lower bound on $b_3(G')$ is obtained by starting with the number of $4$-tuples satisfying Condition 4, subtracting the number of terrible $4$-tuples, and then subtracting the upper bound on the number of bad $4$-tuples.

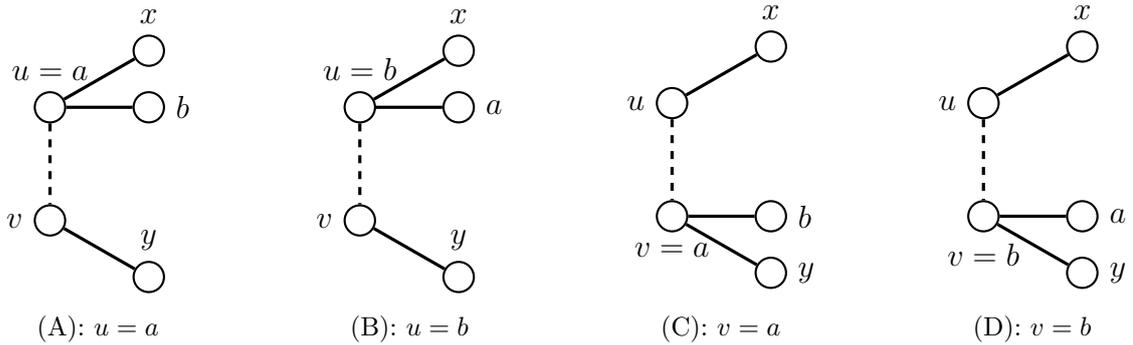
\begin{figure}[htbp]
\centering

\begin{subfigure}[b]{0.24\textwidth}
\centering
\begin{tikzpicture}[
  every edge/.style = {draw=black,very thick},
  vrtx/.style args = {#1/#2}{
    circle, draw, thick, fill=white,
    minimum size=4mm, label=#1:#2
  }
]
\node(u) [vrtx=above/, label=above:{$u=a$}] at (-1.3, 0.75) {};
\node(v) [vrtx=left/$v$] at (-1.3, -0.75) {};
\node(y) [vrtx=above/$y$] at (0, -1.5) {};
\node(x) [vrtx=above/, label=above:{$x$}] at (0, 1.5) {};
\node(a) [vrtx=right/$b$] at (0, 0.75) {};
\path (u) edge[dashed] (v)
      (u) edge (x)
      (u) edge (a)
      (v) edge (y);
\end{tikzpicture}
\caption*{(A): $u = a$}
\end{subfigure}
\begin{subfigure}[b]{0.24\textwidth}
\centering
\begin{tikzpicture}[every edge/.style = {draw=black,very thick},
  vrtx/.style args = {#1/#2}{
    circle, draw, thick, fill=white,
    minimum size=4mm, label=#1:#2
  }
]
\node(u) [vrtx=above/, label=above:{$u=b$}] at (-1.3, 0.75) {};
\node(v) [vrtx=left/$v$] at (-1.3, -0.75) {};
\node(y) [vrtx=above/$y$] at (0, -1.5) {};
\node(x) [vrtx=above/, label=above:{$x$}] at (0, 1.5) {};
\node(a) [vrtx=right/$a$] at (0, 0.75) {};
\path (u) edge[dashed] (v)
      (u) edge (x)
      (u) edge (a)
      (v) edge (y);
\end{tikzpicture}
\caption*{(B): $u = b$}
\end{subfigure}
\begin{subfigure}[b]{0.24\textwidth}
\centering
\begin{tikzpicture}[every edge/.style = {draw=black,very thick},
  vrtx/.style args = {#1/#2}{
    circle, draw, thick, fill=white,
    minimum size=4mm, label=#1:#2
  }
]
\node(u) [vrtx=above/, label=left:{$u$}] at (-1.3, 0.75) {};
\node(v) [vrtx=above/, label=below:{$v = a$}] at (-1.3, -0.75) {};
\node(y) [vrtx=right/$y$] at (0, -1.5) {};
\node(x) [vrtx=above/, label=above:{$x$}] at (0, 1.5) {};
\node(a) [vrtx=right/$b$] at (0, -0.75) {};
\path (u) edge[dashed] (v)
      (u) edge (x)
      (v) edge (a)
      (v) edge (y);
\end{tikzpicture}
\caption*{(C): $v = a$}
\end{subfigure}
\begin{subfigure}[b]{0.24\textwidth}
\centering
\begin{tikzpicture}[every edge/.style = {draw=black,very thick},
  vrtx/.style args = {#1/#2}{
    circle, draw, thick, fill=white,
    minimum size=4mm, label=#1:#2
  }
]
\node(u) [vrtx=above/, label=left:{$u$}] at (-1.3, 0.75) {};
\node(v) [vrtx=above/, label=below:{$v = b$}] at (-1.3, -0.75) {};
\node(y) [vrtx=right/$y$] at (0, -1.5) {};
\node(x) [vrtx=above/, label=above:{$x$}] at (0, 1.5) {};
\node(a) [vrtx=right/$a$] at (0, -0.75) {};
\path (u) edge[dashed] (v)
      (u) edge (x)
      (v) edge (a)
      (v) edge (y);
\end{tikzpicture}
\caption*{(D): $v = b$}
\end{subfigure}

\caption{Terrible cases, satisfying Condition 4 but not Condition 5}
\label{fig:terrible_cases}
\end{figure}

\textbf{Terrible Cases: Exact Count.}
    Please refer to Figure \ref{fig:terrible_cases}, which displays the terrible cases not satisfying Condition 5. There are $d_u d_v$ choices for $x,y$ such that $ux$ and $vy$ are both in $G'$. For each of these choices, we could (A) fix $u = a$ and set any of the $d_u$ neighbours of $u$ in $G'$ to be vertex $b$ and obtain $ab \in G'$ such that $u=a$, where we permit $x = b$. Further, we could (B) set $u = b$ and choose $a$ in $d_u$ ways, where we permit $x = a$. In a similar way, we could choose (C) $v = a$ and set any of the $d_v$ neighbours of $v$ in $G'$ to be vertex $b$, where $b = y$ is permitted. Further, we could set (D) $v = b$ and set any of the $d_v$ neighbours of $v$ in $G'$ to be $a$, where $a = y$ is permitted. This covers all cases of $\{a,b\} \cap \{u,v\} \neq \emptyset$ and shows there are exactly $d_ud_v (2d_u + 2d_u)$ terrible $4$-tuples.

\begin{figure}[htbp]
\centering
\begin{minipage}{0.3\textwidth}
  \centering
  \begin{tikzpicture}[
    every edge/.style = {draw=black,very thick},
    vrtx/.style args = {#1/#2}{%
      circle, draw, thick, fill=white,
      minimum size=4mm, label=#1:#2}
  ]
  \node(u) [vrtx= left/$u$] at (-1.29904, 0.75) {};
  \node(v) [vrtx= left/$v$] at (-1.29904, -0.75) {};
  \node(y) [vrtx= above/$y$] at (0, -1.5) {};
  \node(x) [vrtx= above/, label = above : {$x=a$}] at (0, 1.5) {};
  \node(a) [vrtx= right/$b$] at (1.29904, 0.75) {};

  \path   (u) edge[dashed] (v)
          (u) edge  (x)
          (x) edge  (a)
          (v) edge  (y);
  \end{tikzpicture}

  \parbox{\linewidth}{\centering (A) : $x = a$}
\end{minipage}
\hfill
\begin{minipage}{0.3\textwidth}
  \centering
  \begin{tikzpicture}[
    every edge/.style = {draw=black,very thick},
    vrtx/.style args = {#1/#2}{%
      circle, draw, thick, fill=white,
      minimum size=4mm, label=#1:#2}
  ]
  \node(u) [vrtx= left/$u$] at (-1.29904, 0.75) {};
  \node(v) [vrtx= left/$v$] at (-1.29904, -0.75) {};
  \node(y) [vrtx= above/$y$] at (0, -1.5) {};
  \node(x) [vrtx= above/$x$] at (0, 1.5) {};
  \node(a) [vrtx= right/$a$] at (1.29904, 0.75) {};
  \node(b) [vrtx= right/$b$] at (1.29904, -0.75) {};

  \path   (u) edge[dashed] (v)
          (u) edge   (x)
          (x) edge[red] (a)
          (v) edge  (y)
          (a) edge (b);
  \end{tikzpicture}

  \parbox{\linewidth}{\centering (B) : $xa \in G' \cup M$}
\end{minipage}
\hfill
\begin{minipage}{0.3\textwidth}
  \centering
  \begin{tikzpicture}[
    every edge/.style = {draw=black,very thick},
    vrtx/.style args = {#1/#2}{%
      circle, draw, thick, fill=white,
      minimum size=4mm, label=#1:#2}
  ]
  \node(u) [vrtx= left/$u$] at (-1.29904, 0.75) {};
  \node(v) [vrtx= left/$v$] at (-1.29904, -0.75) {};
  \node(y) [vrtx= above/$y$] at (0, -1.5) {};
  \node(x) [vrtx= above/, label = above : {$x=b$}] at (0, 1.5) {};
  \node(a) [vrtx= right/$a$] at (1.29904, 0.75) {};

  \path   (u) edge[dashed] (v)
          (u) edge  (x)
          (x) edge  (a)
          (v) edge  (y);
  \end{tikzpicture}

  \parbox{\linewidth}{\centering (C) : $x = b$}
\end{minipage}

\caption{Some bad cases, satisfying Conditions 4 and 5 but not Condition 6}
\label{bad_con_5}
\end{figure}
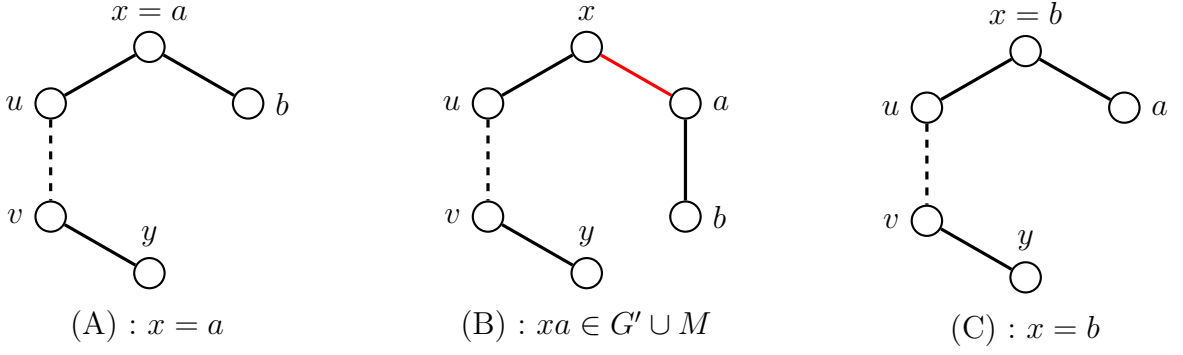

\textbf{Bad Cases: Upper Bound.}
    A bad $4$-tuple with either satisfies Conditions 4 and 5 but not Condition 6, or it will satisfy Conditions 4 and 5 but not Condition 7. Let $(x,a,y,b)$ satisfy Conditions 4 and 5 but not Condition 6. Please refer to Figure \ref{bad_con_5} which illustrates these cases.
    We will upper bound the number of such $4$-tuples. If Condition 6 is not satisfied, then either (A) $x = a$ or (B) $xa$ is an edge in $G'$ or $M$ or (C) $x = b$. We bound all three cases for (A), (B) and (C) simultaneously. There are $d_v$ choices for $y$. For each fixed $y$, the number of $(x,a,y,b)$ such that Conditions 4 and 5 are satisfied but Condition 6 is not satisfied is at most
     \begin{equation}\label{condition_6_bound_complex}
         \sum_{\substack{x \in W \\ ux \in G'}} \Biggl( d_x - 1 + \sum_{\substack{a \in W \\ xa \in G' \cup M \\ a \notin \{u,v\}}} d_a \Biggr).
     \end{equation}
     This count arises as follows. Fix $x \in W$ such that $ux \in G'$. For case (A), there are $d_x - 1$ neighbours of $x$ not equal to $u$ that we could choose to be vertex $b$ (as $u=b$ does not occur if Condition 4 is satisfied). For cases (B) and (C), after fixing $x$ we count the number of directed 2-paths $x \rightarrow a \rightarrow b$ from $x$, where $xa \in G' \cup M$, $ab \in G'$ and $a \notin \{u,v\}$. The number of such 2-paths is at most
     \begin{equation}\label{a_sum}
         \sum_{\substack{a \in W \\ xa \in G' \cup M \\ a \notin \{u,v\}}} d_a.
     \end{equation}
     This is because after fixing $x$ and $a$ such that $xa \in G' \cup M$, there are $d_a$ choices for $b$ resulting in case (B) or (C). None of the choices for $a$ are equal to $u$ or $v$ since Condition 5 is satisfied. Further $a \neq x$ since $a$ is a neighbour of $x$. There are at most $d_x + m_x - 1$ terms in the sum of \eqref{a_sum} since there are $d_x + m_x$ vertices $a$ such that $xa \in G'\cup M$ and we remove the case $a = u$. Therefore, \eqref{a_sum} is bounded above by the sum of the $d_x + m_x - 1$ largest degrees in $G' \setminus M$, where the sum avoids vertices $\{x,u,v\}$. Using this we bound \eqref{condition_6_bound_complex} by
     \begin{align*}
         \sum_{\substack{x \in W \\ ux \in G'}} \Biggl( d_x - 1 + \sum_{\substack{a \in W \\ xa \in G' \cup M \\ a \notin \{u,v\}}} d_a \Biggr) & \leq  \sum_{\substack{x \in W \\ ux \in G'}} \Biggl( d_x - 1 + D^{\{x,u,v\}}(d_x + m_x - 1) \Biggr)  \\
         & \leq  \sum_{\substack{x \in W \\ ux \in G'}} \left(  - 1 + D_H^{\{u,v\}}(d_x + m_x) \right)  & \text{Lemma \ref{deg_sum_prop} (i)} \\
         & = -d_u + \sum_{\substack{x \in W \\ ux \in G'}} D_H^{\{u,v\}}(d_x + m_x). 
     \end{align*}
The above expression is for a fixed $y$. Since there are $d_v$ choices for $y$, the number of bad $4$-tuples satisfying Conditions 4 and 5 but not Condition 6 is at most
\begin{equation}\label{not_6}
    -d_ud_v + d_v \cdot \sum_{\substack{x \in W \\ ux \in G'}} D^{\{u,v\}}(d_x + m_x).
\end{equation}

\begin{figure}[htbp]
\centering
\begin{minipage}{0.3\textwidth}
  \centering
  \begin{tikzpicture}[
    every edge/.style = {draw=black,very thick},
    vrtx/.style args = {#1/#2}{%
      circle, draw, thick, fill=white,
      minimum size=4mm, label=#1:#2}
  ]
  \node(u) [vrtx= left/$u$] at (-1.29904, 0.75) {};
  \node(v) [vrtx= left/$v$] at (-1.29904, -0.75) {};
  \node(y) [vrtx= above/, label = above : {$y=b$}] at (0, -1.5) {};
  \node(x) [vrtx= above/, label = above : {$x$}] at (0, 1.5) {};
  \node(b) [vrtx= right/$a$] at (1.29904, -0.75) {};

  \path   (u) edge[dashed] (v)
          (u) edge  (x)
          (y) edge  (b)
          (v) edge  (y);
  \end{tikzpicture}
  
  \parbox{\linewidth}{\centering (A) : $y = b$}
\end{minipage}
\hfill
\begin{minipage}{0.3\textwidth}
  \centering
  \begin{tikzpicture}[
    every edge/.style = {draw=black,very thick},
    vrtx/.style args = {#1/#2}{%
      circle, draw, thick, fill=white,
      minimum size=4mm, label=#1:#2}
  ]
  \node(u) [vrtx= left/$u$] at (-1.29904, 0.75) {};
  \node(v) [vrtx= left/$v$] at (-1.29904, -0.75) {};
  \node(y) [vrtx= above/$y$] at (0, -1.5) {};
  \node(x) [vrtx= above/$x$] at (0, 1.5) {};
  \node(a) [vrtx= right/$a$] at (1.29904, 0.75) {};
  \node(b) [vrtx= right/$b$] at (1.29904, -0.75) {};

  \path   (u) edge[dashed] (v)
          (u) edge   (x)
          (v) edge  (y)
          (y) edge[red] (b)
          (a) edge (b);
  \end{tikzpicture}
  
  \parbox{\linewidth}{\centering (B) : $yb \in G' \cup M$}
\end{minipage}
\hfill
\begin{minipage}{0.3\textwidth}
  \centering
  \begin{tikzpicture}[
    every edge/.style = {draw=black,very thick},
    vrtx/.style args = {#1/#2}{%
      circle, draw, thick, fill=white,
      minimum size=4mm, label=#1:#2}
  ]
  \node(u) [vrtx= left/$u$] at (-1.29904, 0.75) {};
  \node(v) [vrtx= left/$v$] at (-1.29904, -0.75) {};
  \node(y) [vrtx= above/, label = above : {$y=a$}] at (0, -1.5) {};
  \node(x) [vrtx= above/, label = above : {$x$}] at (0, 1.5) {};
  \node(a) [vrtx= right/$b$] at (1.29904, -0.75) {};

  \path   (u) edge[dashed] (v)
          (u) edge  (x)
          (y) edge  (a)
          (v) edge  (y);
  \end{tikzpicture}
  
  \parbox{\linewidth}{\centering (C) : $y = a$}
\end{minipage}

\caption{Conditions 4 and 5 are satisfied but not Condition 7.}
\label{bad_con_7}
\end{figure}
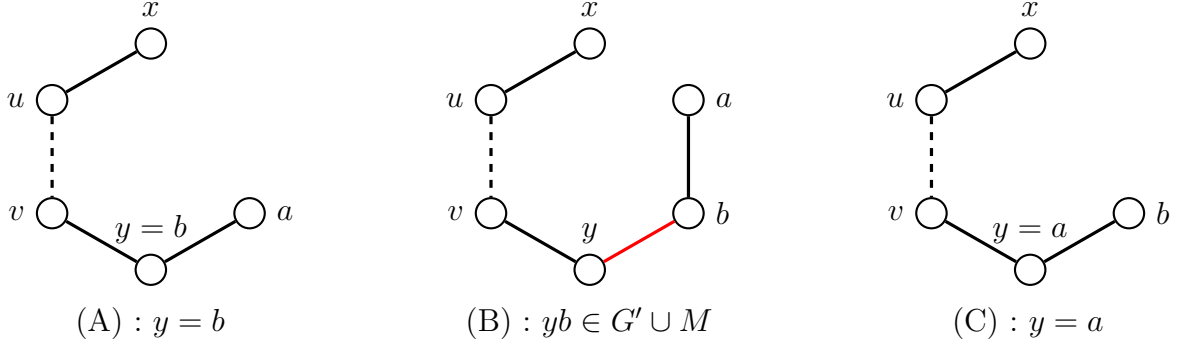

Now we upper bound the number of bad cases satisfying Conditions 4 and 5 but not satisfying Condition 7. Figure \ref{bad_con_7} displays these cases.
By a symmetric argument, the number of bad $4$-tuples satisfying Conditions 4 and 5 but not Condition 7 is at most
\begin{equation}\label{not_7}
    -d_ud_v + d_u \cdot \sum_{\substack{y \in W \\ vy \in G'}} D^{\{u,v\}}(d_y + m_y).
\end{equation}
Therefore, the total number of bad and terrible $4$-tuples is at most the sum of the terrible $4$-tuples with \eqref{not_6} and \eqref{not_7}. This is
\begin{align*}
     d_u d_v(2d_u + 2d_v) & + \Biggl( -d_ud_v + d_v \cdot \sum_{\substack{x \in W \\ ux \in G'}} D^{\{u,v\}}(d_x + m_x)\Biggr)   \\
     & + \Biggl(-d_u d_v + d_u \cdot \sum_{\substack{y \in W \\ vy \in G'}} D^{\{u,v\}}(d_y + m_y)\Biggr).
\end{align*}
To simplify the above, we distribute the first expression into the two summations. We obtain
\begin{align*}
     d_u^2d_v + d_ud_v^2 & + \Biggl( -d_ud_v + d_v \cdot \sum_{\substack{x \in W \\ ux \in G'}} D^{\{u,v\}}(d_x + m_x)\Biggr)   \\
     {}+ d_u^2d_v + d_ud_v^2  & + \Biggl(-d_ud_v + d_u \cdot \sum_{\substack{y \in W \\ vy \in G'}} D^{\{u,v\}}(d_y + m_y)\Biggr) \numberthis \label{to_simplify}.
\end{align*}
Let us consider the first line of the above expression (the reasoning for the second line follows by a symmetric argument). We have
{\allowdisplaybreaks
\begin{align*}
    d_u^2d_v &+ d_ud_v^2  + \Biggl( -d_ud_v + d_v \cdot \sum_{\substack{x \in W \\ ux \in G'}} D^{\{u,v\}}(d_x + m_x)\Biggr) \\
    & = d_v\sum_{\substack{x \in W \\ ux \in G'}} \Biggl(d_u + d_v  -1 +  D^{\{u,v\}}(d_x + m_x) \Biggr) & \text{$d_u$ terms in the sum} \\
    & \leq d_v\sum_{\substack{x \in W \\ ux \in G'}} \left(-1 +  D(d_x + m_x+2) \right) & \text{Lemma \ref{deg_sum_prop} (i)} \\
    & \leq d_v d_u D(\alpha_u(\db,\mb,\db,\zerob) + 2).
\end{align*}}
The last step follows by Lemma \ref{deg_sum_prop} (ii) and Jensen's inequality. By a symmetric argument, we can upper bound \eqref{to_simplify} and hence the number of bad and terrible $4$-tuples by $d_v d_u \left(D(\alpha_u(\db,\mb,\db,\zerob)+2) + D(\alpha_v(\db,\mb,\db,\zerob)+2) \right)$. Therefore, a lower bound for $b_3(G')$ is
\begin{align*}
    b_3(G') & \geq 2m(G)d_u d_v - d_v d_u \left( D(\alpha_u(\db,\mb,\db,\zerob)+2) + D(\alpha_v(\db,\mb,\db,\zerob)+2) \right) \\
    & = 2m(G) d_u d_v \cdot g(\db,\emptygraph,M,uv). \qedhere
\end{align*}
\end{proof}

Let $T'$ denote the number of pairs $(G,G')$, where $G \in \Gs_{\db}(uv,M)$ and $G' \in \Gs_{\db}(\emptygraph,M+uv)$ and there exists a backward 3-switching from $G'$ to $G$. Observe that this implies there is also a forward 3-switching from $G$ to $G'$, and we have the equality
\begin{equation*}
 \sum_{G' \in \Gs_{\db}(\emptygraph,M+uv)} b_3(G') = T' = \sum_{G \in \Gs_{\db}(uv,M)} f_3(G) . 
\end{equation*}
Claim \ref{forward_3_switch_upper} and Claim
\ref{backward_switch_3_lower_general} then give
\begin{equation*}
    |\Gs_{\db}(\emptygraph,M + uv)| \cdot 2m(G) d_u d_v \cdot g(\db,\emptygraph,M,uv)  \leq |\Gs_{\db}(uv,M)| \cdot (2m(G))^2 . 
\end{equation*}
Observe that the substitutions $\db \mapsto \db - \hb$ and $M \mapsto H \cup L$ give the result as shown in Lemma \ref{ratio_bound_lemma}(B). Indeed under the substitutions, $m(G) \mapsto m(G) - m(H)$ and $g(\db,\emptygraph,M,uv) \mapsto g(\db,H,L,uv)$.
\end{proof}

\subsection{Proof of Corollaries \ref{multiple_edge_cor} and \ref{multiple_forbidden_cor}}
\begin{proof}[Proof of Corollary \ref{multiple_edge_cor}]
Choose vertices $u_1,\dots,u_{m(X)},v_1,\dots,v_{m(X)}$ such that $e_i = u_iv_i$ and $\{e_i\}_{i=1}^{m(X)}$ is an enumeration of the edges in $X$ as given in Theorem \ref{multiple_edge_thm}. We will set assumptions on the functions $f$ and $g$ present there in order to simplify the products.

\textbf{Part A: Upper Bound}.
For $k$ large, using Lemma \ref{deg_sum_prop} (ii) we have
\begin{equation*}
    \frac{D_{X_{i-1}}(d_{u_i} + l_{u_i}) + D_{X_{i-1}}(d_{v_i} + l_{v_i})}{2m(G) - 2m(X_{i-1})} \leq \frac{D(\varDelta_{\partial X}(\db + \lb))}{m(G) - m(X) } \leq \rho(k).
\end{equation*}
Theorem \ref{multiple_edge_thm}(A) then applies. We now bound the product of the terms $f(\db,X_{i-1},L,e_i)$ present there. Letting $C(k) \coloneq \frac{\ln(1-\rho(k))}{\rho(k)}$ and using $1-x \geq e^{C(k) x}$ for $x \in [0,\rho(k)]$ gives
\begin{align*}
    \prod_{i=1}^{m(X)}  f(\db, X_{i-1}, L, e_i) & \geq \exp\biggl(C(k) \cdot \sum_{i=1}^{m(X)} \frac{D_{X_{i-1}}(d_{u_i} + l_{u_i}) + D_{X_{i-1}}(d_{v_i} + l_{v_i}) }{2m(G) - 2m(X_{i-1})}\biggr)  \\
    & \geq \exp\left(C(k) \cdot \frac{m(X)\cdot D(\kappa) }{m(G) - m(X)}\right) \numberthis\label{kappa_expression}  \\
    & \geq \exp\left(C(k) \cdot \frac{m(X)\cdot D(\varDelta_{\partial X}(\db + \lb)) }{m(G) - m(X)}\right), \numberthis\label{final_upper_term} 
\end{align*}
where $\kappa \coloneq \frac{\sum_{j \in W} x_{j}(d_j + l_j)}{2m(X) } \leq \varDelta_{\partial X}( \db + \lb)$.
The second last inequality follows by Lemma \ref{deg_sum_prop} (iii), observing $\sum_{i=1}^{m(X)}  D(d_{u_i} + l_{u_i}) + D(d_{v_i} + l_{v_i}) = \sum_{j \in W} x_{j} \cdot D(d_j + l_j)$, and then applying Jensen's inequality to the degree sum function. As $\rho(k)$ is bounded away from $1$, $C(k)$ is bounded (note that $C(k) \to -1$ if $\rho(k) \to 0$). The assumptions in the corollary then imply \eqref{final_upper_term} is either in $1+o(1)$ or in $O(1)$.

\textbf{Part B: Lower Bound.}
Let $\xb_i$ be the degree sequence of $X_{i}$. Observe that for all $1 \leq i \leq m(X)$ and all $w \in W$, $\alpha_w(\db,\lb,\db,\xb_{i-1}) \leq \gamma(\db,\lb,\db,\xb)$. Consequently, for $k$ large, using Lemma \ref{deg_sum_prop} (ii) gives
\begin{equation*}
    \frac{D_{X_{i-1}}(\alpha_{u_i}(\db,\lb,\db,\xb_{i-1})+2)+D_{X_{i-1}}(\alpha_{v_i}(\db,\lb,\db,\xb_{i-1})+2)}{2m(G) - 2m(X_{i-1})} \leq \frac{D(\gamma(\db,\lb,\db,\xb)+2)}{m(G) - m(X)} \leq \rho(k).
\end{equation*}
Theorem \ref{multiple_edge_thm}(B) then applies. We now bound the product of the terms $g(\db,X_{i-1},L,e_i)$ present there. Let $C(k) \coloneq \frac{\ln(1-\rho(k))}{\rho(k)}$. Again, using $1-x \geq e^{C(k) x}$ for $x \in [0,\rho(k)]$ gives
\begin{align*}
     \prod_{i=1}^{m(X)} & g(\db, X_{i-1}, L, e_i) \\[-2ex]
    &  \geq \exp\biggl(C(k) \cdot \sum_{i=1}^{m(X)}  \frac{D_{X_{i-1}}(\alpha_{u_i}(\db,\lb,\db,\xb_{i-1})+2) + D_{X_{i-1}}(\alpha_{v_i}(\db,\lb,\db,\xb_{i-1})+2)}{2m(G) - 2m(X_{i-1})}\biggr) \\
    &  \geq \exp\left(C(k) \cdot \frac{m(X) \cdot D(\mu+2)}{m(G) - m(X)}\right)\\
    &  \geq \exp\left(C(k) \cdot \frac{m(X) \cdot D(\gamma(\db,\lb,\db,\xb)+2)}{m(G) - m(X)}\right), \numberthis \label{final_lower_term}
\end{align*}
where $\mu \coloneq \frac{\sum_{i \in W} x_{i}\alpha_{i}(\db,\lb,\db,\xb) }{2m(X)} \leq \gamma(\db,\lb,\db,\xb)$. The second inequality follows by Lemma \ref{deg_sum_prop} (iii), observing that $\sum_{i=1}^{m(X)} \!D(\alpha_{u_i}(\db,\lb,\db,\xb)+2) + D(\alpha_{v_i}(\db,\lb,\db,\xb)+2) = \sum_{i \in W} x_{i} D(\alpha_{i}(\db,\lb,\db,\xb)+2)$, and then using Jensen's inequality. Since $\rho(k)$ is bounded away from $1$, $C(k)$ is bounded. The results follow from Theorem \ref{multiple_edge_thm}(B) as \eqref{final_lower_term} is either in $1+o(1)$ or in $\Omega(1)$ under the assumptions, and $\phi(\db,X)^{m(X)}$ is either in $1+o(1)$ or in $O(1)$ using the assumption on $\varLambda(k)$ and the inequality $1+x \leq e^x$.
\end{proof}

\begin{proof}[Proof of Corollary \ref{multiple_forbidden_cor}]
We again set assumptions to control the products of the terms $g$ and $f$ present in Theorem \ref{multiple_forbidden_theorem}. Let $\lb_j$ be the degree sequence of $L_j$.

\textbf{Part A: Upper Bound.}
 Let $C(k) = \frac{\ln(1-\rho(k))}{\rho(k)}$. Observe that for all $1 \leq j\leq m(Y)$ and all $w \in W$, $\alpha_w(\db,\lb_{j-1},\db,\zerob) \leq \gamma(\db,\lb + \yb,\db,\zerob)$. Consequently,
\begin{equation*}
    \frac{D(\alpha_{p_j}(\db,\lb_{j-1},\db,\zerob)+2)+D(\alpha_{q_j}(\db,\lb_{j-1},\db,\zerob)+2)}{2m(G)} \leq \frac{D(\gamma(\db,\lb + \yb,\db,\zerob)+2)}{m(G)} \leq \rho(k).
\end{equation*}
    Theorem \ref{multiple_forbidden_theorem}(A) then applies. We now bound the product of the terms $g(\db,\emptygraph,L_{j-1},\ebar_j)$ present in that theorem. Using $1-x \geq e^{C(k) x}$ for $x \in [0,\rho(k)]$ gives
\begin{align*}
    &\prod_{j=1}^{m(Y)} g(\db, \emptygraph,L_{j-1},\ebar_j)  \\[-2ex]
    &\qquad \geq \exp\biggl( C(k) \cdot \sum_{j=1}^{m(Y)} \frac{D(\alpha_{p_j}(\db,\lb_{j-1},\db,\zerob)+2) + D(\alpha_{q_j}(\db,\lb_{j-1},\db,\zerob)+2)}{2m(G)} \biggr) \\
    &\qquad \geq \exp\left( C(k) \cdot \frac{m(Y) D(\lambda + 2)}{m(G)} \right) \\
    &\qquad \geq \exp\left( C(k) \cdot \frac{m(Y) D(\gamma(\db,\lb + \yb,\db,\zerob) + 2)}{m(G)} \right), \numberthis\label{forbidden_upper_bound}
\end{align*}
where $\lambda \coloneq \frac{\sum_{i\in W} y_{i} \alpha_i(\db,\lb + \yb,\db,\zerob)}{2m(Y)} \leq \gamma(\db,\lb + \yb,\db,\zerob)$. The second last inequality follows by observing $\sum_{j=1}^{m(Y)}  D(\alpha_{p_j}(\db,\lb + \yb,\db,\zerob)+2) + D(\alpha_{q_j}(\db,\lb + \yb,\db,\zerob)+2) = \sum_{i \in W} y_{i} D(\alpha_{i}(\db,\lb + \yb,\db,\zerob)+2)$, and then applying Jensen's inequality. Since $\rho(k)$ is bounded away from $1$, $C(k)$ is bounded. The assumptions of the corollary imply that \eqref{forbidden_upper_bound} is either in $1+o(1)$ or in $O(1)$.

\textbf{Part B: Lower Bound.}

Let $C(k) = \frac{\ln(1-\rho(k))}{\rho(k)}$. Note that for all $1\leq j\leq m(Y)$,
\begin{equation*}
    \frac{D(d_{p_j} + l_{j-1,p_j}) + D(d_{q_j} + l_{j-1,q_j}) }{2m(G)} \leq \frac{ D(\varDelta_{\partial Y}(\db + \lb + \yb))}{m(G)} \leq \rho(k).
\end{equation*}
Theorem \ref{multiple_forbidden_theorem}(B) then  applies. We now bound the product of the terms $f(\db,\emptygraph,L_{j-1},\ebar_j)$ present there. Using $1-x \geq e^{C(k) x}$ for $x \in [0,\rho(k)]$ gives
\begin{align*}
     \prod_{j=1}^{m(Y)} f(\db,\emptygraph,L_{j-1},p_iq_i) & \geq \exp\Biggl(\! C(k) \cdot  \sum_{j=1}^{m(Y)} \frac{D(d_{p_j} + l_{j-1,p_j}) + D(d_{q_j} + l_{j-1,q_j}) }{2m(G)} \Biggr) \\
     & \geq \exp\biggl(\! C(k) \cdot \frac{m(Y) D(\eta)}{m(G)} \biggr) \\
     & \geq \exp\biggl(\! C(k) \cdot \frac{m(Y) D(\varDelta_{\partial Y}(\db + \lb + \yb))}{m(G)} \biggr), \numberthis\label{lower_bound_forbidden}
\end{align*}
where $\eta \coloneq \frac{\sum_{i \in W} y_{i}(d_{i} + l_{i} + y_i)}{2m(Y)} \leq \varDelta_{\partial Y}(\db + \lb + \yb)$. The second last inequality follows by Lemma \ref{deg_sum_prop} (ii), observing $\sum_{j=1}^{m(Y)} D(d_{p_j} + l_{m(Y),p_j}) + D(d_{q_j} + l_{m(Y),q_j}) = \sum_{i \in W} y_{i} D(d_{i} + l_{m(Y),i})$, and then applying Jensen's inequality. Again, $C(k)$ is bounded since $\rho(k)$ is bounded. The results now follow from the assumptions of the theorem, as \eqref{lower_bound_forbidden} is either in $1+o(1)$ or in $\Omega(1)$ depending on the behaviour of $m(Y) \cdot \rho(k)$.
\end{proof}

\section{Proof of Bipartite Cases: Lemma \ref{single_edge_prob_lemma_bipartite} and Theorems \ref{multiple_edges_bipartite}, \ref{multiple_forbidden_theorem_bipartite}}\label{bipartite_proofs_section}
The proofs are similar to the generic case. Indeed the main difference is accounting for the bipartition in the switching argument of Lemma \ref{ratio_bound_lemma_bipartite}.

\begin{lemma}\label{ratio_bound_lemma_bipartite}
    Let $H,L \subseteq K_{m,n}$ with $H \cap L = \emptygraph$. Let $(\hb,\ib)$ be the degree sequence of $H$. Let $uv \notin H \cup L$ be an edge of $K_{m,n}$.
    \begin{enumerate}[label=\textbf{(\Alph*)}, leftmargin=*]
    \item   If $p(\sbb,\tb,H,L,uv) >0$, then
    \begin{equation*}
        \frac{|\Bs_{(\sbb,\tb)}(H+uv,L)|}{|\Bs_{(\sbb,\tb)}(H,L+uv)|} \leq \frac{(s_u-h_u)(t_v-i_v)}{m(G) - m(H)} \cdot \frac{1}{p(\sbb,\tb,H,L,uv) }.
    \end{equation*}
    \item It always holds that
    \begin{equation*}
         \frac{|\Bs_{(\sbb,\tb)}(H+uv,L)|}{|\Bs_{(\sbb,\tb)}(H,L+uv)|}  \geq \frac{(s_u-h_u)(t_v-i_v)}{m(G) - m(H)} \cdot q(\sbb,\tb, H, L, uv).   
    \end{equation*}
\end{enumerate}
The inequality of Part (B) always holds, but is only useful for probability estimates if we assume $q(\sbb,\tb,H,L,uv) > 0$, since this yields valid probabilities in $[0,1]$.
\end{lemma}
The proof of Lemma \ref{ratio_bound_lemma_bipartite} is left to the end of this section. Firstly, we will sketch the proof of Lemma \ref{single_edge_prob_lemma_bipartite} and Theorems \ref{multiple_edges_bipartite}, \ref{multiple_forbidden_theorem_bipartite}.
\begin{proof}[Proof Sketch: Lemma \ref{single_edge_prob_lemma_bipartite} and Theorem \ref{multiple_edges_bipartite}]
    The proof is extremely similar to the generic case, but we use the bounds presented in Lemma \ref{ratio_bound_lemma_bipartite} combined with
    \begin{equation*}
        \Pa(X \subseteq G) = \prod_{i=1}^{m(X)}\left(1 + \frac{|\Bs_{(\sbb,\tb)}(X_{i-1},L+e_i)|}{|\Bs_{(\sbb,\tb)}(X_{i-1}+e_i,L)}|\right)^{-1}. \qedhere
    \end{equation*}
\end{proof}

\begin{proof}[Proof of Theorem \ref{multiple_forbidden_theorem_bipartite}]
    The proof is extremely similar to the generic case except we use the bounds presented in Lemma \ref{ratio_bound_lemma_bipartite} combined with
\begin{equation*}
    \Pa(Y \cap G = \emptygraph) = \prod_{j=1}^{m(Y)} \left(1 + \frac{|\Bs_{(\sbb,\tb)}(\ebar_j,L_{j-1})|}{|\Bs_{(\sbb,\tb)}(\emptygraph,L_{j-1}+\ebar_j)|}\right)^{-1}. \qedhere
\end{equation*}
\end{proof}

\subsection{Proof of Lemma \ref{ratio_bound_lemma_bipartite}}
Let $M \subseteq K_{m,n}$ be a bipartite graph with degree sequence $(\mb,\nb)$ and $uv \in K_{m,n}$ be an edge with $uv \notin M$. We will use the method of switchings to estimate the ratio $\frac{|\Bs_{(\sbb,\tb)}(uv,M)|}{|\Bs_{(\sbb,\tb)}(\emptygraph,M+uv)|}$, where $(\sbb,\tb)$ is a bipartite degree sequence with sizes $m$ and $n$. Again, this is sufficient as we can perform the substitutions $(\sbb,\tb) \mapsto (\sbb,\tb) - (\hb,\ib)$ and $M \mapsto H \cup L$ 
 to obtain the estimates presented in Lemma \ref{ratio_bound_lemma_bipartite}, where $(\hb,\ib)$ is the degree sequence of $H$. The calculations are similar to the generic case, but we preserve the bipartition during the switching, so certain cases do not occur. Again, solid black lines are edges in $G$ and dashed black lines are non-edges. Red edges are known to be in $G \cup M$. Vertices in $S$ are white and vertices in $T$ black.

\subsection*{Part A Proof: 2-Switchings}
\subsubsection*{Forward 2-Switchings}

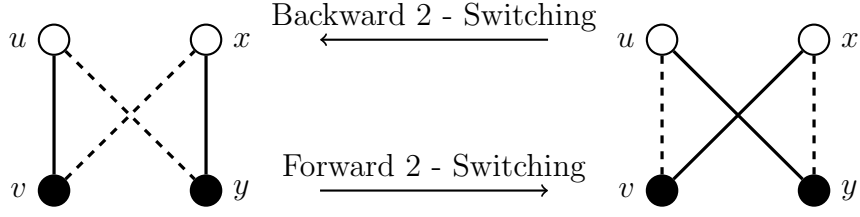
\begin{figure}[htbp]
\centering
   \begin{tikzpicture}[
every edge/.style = {draw=black,very thick},
 vrtx/.style args = {#1/#2}{%
      circle, draw, thick, fill=white,
      minimum size=4mm, label=#1:#2}
                    ]
\node(u) [vrtx= left/$u$] at (0-5, 2) {};
\node(v) [vrtx= left/$v$,fill = black] at (0-5, 0) {};
\node(y) [vrtx= right/$y$,fill = black] at (2-5, 0) {};
\node(x) [vrtx= right/$x$] at (2-5, 2) {};

\path   (u) edge (v)
        (x) edge (y)
        (u) edge[dashed]   (y)
        (v) edge[dashed]   (x);

\node(u') [vrtx= left/$u$] at (0+3, 2) {};
\node(v') [vrtx= left/$v$,fill = black] at (0+3, 0) {};
\node(y') [vrtx= right/$y$,fill = black] at (2+3, 0) {};
\node(x') [vrtx= right/$x$] at (2+3, 2) {};

\path   (u') edge[dashed] (v')
        (x') edge[dashed] (y')
        (u') edge  (y')
        (v') edge  (x');

\node(forward) at (-0,0.3) (<name>) {Forward 2 - Switching};
\draw[->, above, thick] (-1.5,0) -- (1.5,0);

\node(back) at (0,0.3+2) (<name>) {Backward 2 - Switching};
\draw[<-, above, thick] (-1.5,0+2) -- (1.5,0+2);
\end{tikzpicture} 
\caption{A 2-switching between $\Bs_{(\sbb,\tb)}(uv,M)$ and $\Bs_{(\sbb,\tb)}(\emptygraph,M+uv)$}
\label{2_switch_bipartite}
\end{figure} 

\begin{defn}
    Given $G \in \Bs_{(\sbb,\tb)}(uv,M)$, we call an ordered $2$-tuple $(x,y) \in S \times T$ a \textit{forward 2-switching} on $G$ if the following conditions hold:
    \begin{enumerate}
        \item $xy$ is an edge in $G$ and $xy \neq uv$.
        \item $x \neq u$, $y \neq v$ and $uy,xv$ are not edges in $G \cup M$.
    \end{enumerate}
\end{defn}
Given a forward 2-switching $(x,y)$, we can delete the edges $uv,xy$ and add the edges $uy,xv$ to obtain a new bipartite graph $G' \in \Bs_{(\sbb,\tb)}(\emptygraph,M+uv)$. Let $f_2^B(G)$ denote the number of forward $2$-switchings from $G$ to $G'$.
\begin{claim}\label{forward_2_switch_lower_bipartite}
For all $G \in \Bs_{(\sbb,\tb)}(uv,M)$,
    \begin{equation*}
        f_2^B(G) \geq m(G) \cdot p(\sbb,\tb,\emptygraph,M,uv)
    \end{equation*}
\end{claim}
\begin{proof}
    The proof is similar to the generic case proof of Claim \ref{forward_2_switch_lower}. Let $(\mb,\nb)$ be the degree sequence of $M$. There are $m(G) - 1$ ordered pairs $(x,y)$ such that Condition 1 holds. Figure \ref{bad_forward_2_bipartite} shows the cases where Condition 1 holds but Condition 2 fails. There are at most $S(t_v + n_v) + T(s_u + m_u)-2$ such $2$-tuples. Inclusion-exclusion then implies the result.
    \begin{figure}[htbp]
\centering
    \begin{tabularx}{0.95\textwidth}{*{3}{>{\centering\arraybackslash}X}}
   \begin{tikzpicture}[
every edge/.style = {draw=black,very thick},
 vrtx/.style args = {#1/#2}{%
      circle, draw, thick, fill=white,
      minimum size=4mm, label=#1:#2}
                    ]
\node(u) [vrtx= above/, label = above : {$u=x$}] at (0, 2) {};
\node(v) [vrtx= left/$v$,fill = black] at (0, 0) {};
\node(y) [vrtx= right/$y$,fill = black] at (2, 0) {};

\path   (u) edge (v)
        (u) edge   (y);        
        \end{tikzpicture} 
    \centerline{(A) : $u = x$}  
    &   
      \begin{tikzpicture}[
every edge/.style = {draw=black,very thick},
 vrtx/.style args = {#1/#2}{%
      circle, draw, thick, fill=white,
      minimum size=4mm, label=#1:#2}
                    ]
\node(u) [vrtx= left/$u$] at (0, 2) {};
\node(v) [vrtx= left/$v$,fill = black] at (0, 0) {};
\node(y) [vrtx= right/$y$,fill = black] at (2, 0) {};
\node(x) [vrtx= right/$x$] at (2, 2) {};

\path   (u) edge (v)
        (x) edge (y)
        (u) edge[red]   (y);
\end{tikzpicture} 
\centerline{(B) : $uy \in G \cup M$\vrule width0pt depth3ex} \\

 \begin{tikzpicture}[
every edge/.style = {draw=black,very thick},
 vrtx/.style args = {#1/#2}{%
      circle, draw, thick, fill=white,
      minimum size=4mm, label=#1:#2}
                    ]
\node(u) [vrtx= left/$u$] at (0, 2) {};
\node(v) [vrtx= left/, label = left : {$v=y$},fill = black] at (0, 0) {};
\node(y) [vrtx= right/$x$] at (2, 2) {};

\path   (u) edge (v)
        (v) edge   (y);
        \end{tikzpicture} 
    \centerline{(C) : $v = y$} 
        & 
        \begin{tikzpicture}[
every edge/.style = {draw=black,very thick},
 vrtx/.style args = {#1/#2}{%
      circle, draw, thick, fill=white,
      minimum size=4mm, label=#1:#2}
                    ]
\node(u) [vrtx= left/$u$] at (0, 2) {};
\node(v) [vrtx= left/$v$,fill = black] at (0, 0) {};
\node(y) [vrtx= right/$y$,fill = black] at (2, 0) {};
\node(x) [vrtx= right/$x$] at (2, 2) {};

\path   (u) edge (v)
        (x) edge (y)
        (v) edge[red]   (x);
        \end{tikzpicture} 
    \centerline{(D) : $xv \in G \cup M$} 
    \end{tabularx}
\caption{Cases where Condition 2 is not satisfied}
\label{bad_forward_2_bipartite}
\end{figure}
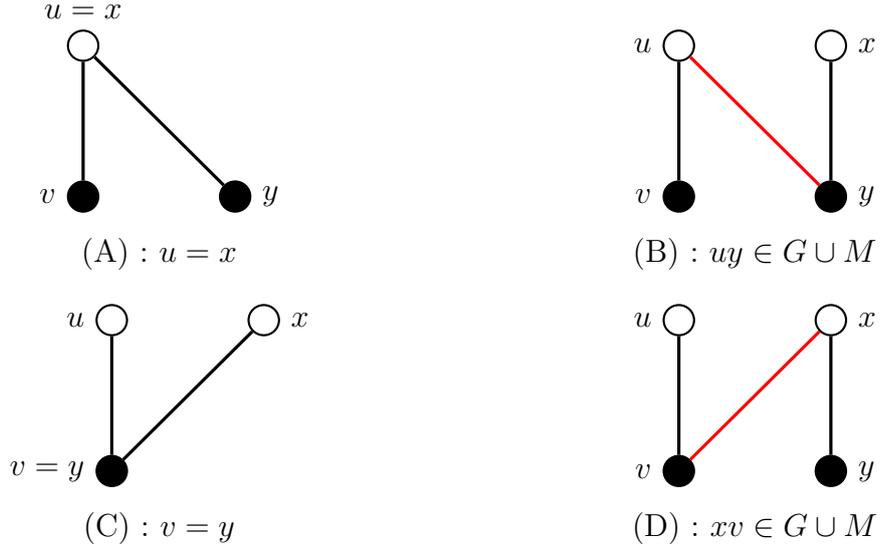 
\end{proof}

\subsubsection*{Backward 2-Switchings}
\begin{defn}[Backward 2-Switching]
    Given $G' \in \Bs_{(\sbb,\tb)}(\emptygraph,M+uv)$, we call an ordered $2$-tuple $(x,y) \in S \times T$ a \textit{backward 2-switching} on $G'$ if the following conditions hold:
    \begin{enumerate}
        \setcounter{enumi}{2}
        \item $uy$ and $xv$ are edges in $G'$.
        \item $x \neq y$ and $xy$ is not an edge in $G'$ or $M$.
    \end{enumerate}
\end{defn}
Given a backward 2-switching on $G'$, we can delete the edges $uy$ and $xv$ and add the edges $uv$ and $xy$ to obtain a new bipartite graph $G \in \Bs_{(\sbb,\tb)}(uv,M)$. Let $b_2^B(G')$ denote the number of backward 2-switchings from $G'$. We now upper bound this quantity.
\begin{claim}\label{bipartite_2_backward_upper}
    For all $G' \in \Bs_{(\sbb,\tb)}(\emptygraph,M+uv)$,
    \begin{equation*}
        b_2^B(G') \leq s_u t_v.
    \end{equation*}
\end{claim}
\begin{proof}
    There are exactly $s_ut_v$ ordered pairs $(x,y)$ satisfying Condition 3.
\end{proof}

Using the same counting argument as in the generic case, let $T$ denote the number of pairs $(G,G')$, where $G \in \Bs_{(\sbb,\tb)}(uv,M)$ and $G' \in \Bs_{(\sbb,\tb)}(\emptygraph,M+uv)$ and there exists a forward switching from $G$ to $G'$. Since this implies there exists a backward switching from $G'$ to $G$, we obtain
\begin{equation*}
 \sum_{G \in \Bs_{(\sbb,\tb)}(uv,M)} f_2^B(G) = T = \sum_{G' \in \Bs_{(\sbb,\tb)}(\emptygraph,M+uv)} b_2^B(G'). 
\end{equation*}
Claim \ref{forward_2_switch_lower_bipartite} and Claim \ref{bipartite_2_backward_upper} then give
\begin{equation*}
    m(G) \cdot p(\sbb,\tb,\emptygraph,M,uv) \cdot |\Bs_{(\sbb,\tb)}(uv,M)| \leq s_u t_v \cdot |\Bs_{(\sbb,\tb)}(\emptygraph,M+uv)|.
\end{equation*}
The assumption $p(\sbb,\tb,\emptygraph,M,uv) > 0$ allows one to rearrange the above. The substitution of $(\sbb,\tb) \mapsto (\sbb,\tb) - (\hb,\ib)$ and $M \mapsto H \cup L$ gives the result.
\subsection*{Part B Proof: 3-Switchings}
\subsubsection*{Forward 3-Switchings}
\begin{figure}[htbp]
\centering
   \begin{tikzpicture}[scale=0.9,
every edge/.style = {draw=black,very thick},
 vrtx/.style args = {#1/#2}{%
      circle, draw, thick, fill=white,
      minimum size=4mm, label=#1:#2}
                    ]
\node(u) [vrtx= left/$u$] at (-8, 2) {};
\node(v) [vrtx= left/$v$,fill = black] at (-8, 0) {};
\node(w) [vrtx= above/$y$] at (-6, 2) {};
\node(c) [vrtx= below/$x$,fill = black] at (-6, 0) {};
\node(a) [vrtx= right/$a$] at (-4,2) {};
\node(b) [vrtx= right/$b$,fill = black] at (-4,0) {};

\path   (u) edge (v)
        (u) edge[dashed]   (c)
        (c) edge (a)
        (b) edge (w)
        (v) edge[dashed]   (w)
        (a) edge[dashed]   (b);

\node(u') [vrtx= left/$u$] at (0+2, 2) {};
\node(v') [vrtx= left/$v$, fill = black] at (0+2, 0) {};
\node(w') [vrtx= above/$y$] at (2+2, 2) {};
\node(c') [vrtx= below/$x$, fill = black] at (2+2, 0) {};
\node(a') [vrtx= right/$a$] at (4+2,2) {};
\node(b') [vrtx= right/$b$, fill = black] at (4+2,0) {};

\path   (u') edge[dashed]  (v')
        (u') edge  (c')
        (c') edge[dashed]  (a')
        (b') edge[dashed]  (w')
        (v') edge  (w')
        (a') edge  (b');

\node(forward) at (-1,0.3) (<name>) {Forward 3-Switching};
\draw[->, above, thick] (-2.5,0) -- (0.5,0);

\node(back) at (-1,0.3+2) (<name>) {Backward 3-Switching};
\draw[<-, above, thick] (-2.5,0+2) -- (0.5,0+2);

\end{tikzpicture} 
\caption{A switching between $\Bs_{(\sbb,\tb)}(uv,M)$ and $\Bs_{(\sbb,\tb)}(\emptygraph,M+uv)$}
\label{3_switch_bipartite}
\end{figure} 
Please refer to Figure \ref{3_switch_bipartite}.
\begin{defn}[Forward 3-Switching]
    Given $G \in \Bs_{(\sbb,\tb)}(uv,M)$, we call an ordered $4$-tuple $(y,a,x,b) \in S^2 \times T^2$ a \textit{forward 3-switching} on $G$ if the following conditions hold:
\begin{enumerate}
    \item $uv,ax,yb$ are edges in $G$ (not necessarily distinct).
    \item The vertices $u,v,a,x,y,b$ are all distinct.
    \item $ux,yv,ab$ are not edges in $G$ or $M$.
\end{enumerate}
\end{defn}
If we can find a forward 3-switching on $G$, then by deleting the edges $uv,ax$ and $yb$ and adding the edges $ux,yb$ and $ab$ we obtain a graph $G' \in \Bs_{(\sbb,\tb)}(\emptygraph,M+uv)$. Condition~2 ensures $uv,ux,yv,yb,ax,ab$ are edges or non-edges in $G$ and they are distinct. Condition~3 ensures $ux,yv$ and $ab$ can be added in the switching. Let $f_3^B(G)$ denote the number of forward 3-switchings from $G$.
\begin{claim}\label{forward_3_switch_upper_bipartite}
    For all $G \in \Bs_{(\sbb,\tb)}(uv,M)$,
    \begin{equation*}
        f_3^B(G) \leq m(G)^2.
    \end{equation*}
\end{claim}
\begin{proof}
    The number of $4$-tuples satisfying Condition 1 is at most $m(G)^2$.
\end{proof}

\subsubsection*{Backward 3-Switchings}

\begin{defn}[Backward 3-Switching]
    Given $G' \in \Bs_{(\sbb,\tb)}(\emptygraph, M + uv)$, we call an ordered $4$-tuple $(y,a,x,b) \in S^2 \times T^2$ a \textit{backward 3-switching} on $G'$ if the following conditions hold:
    \begin{enumerate}
    \setcounter{enumi}{3}
    \item $ux,yv,ab$ are edges in $G'$ (not necessarily different).
    \item $a \neq u$ and $b \neq v$.
    \item $x \neq b $ and $ax$ is not an edge in $G'$ or $M$.
    \item $y \neq a$ and $yb$ is not an edge in $G'$ or $M$.
\end{enumerate}
\end{defn}
Given a backward 3-switching on $G'$, we can delete the edges $ux,yv$ and $ab$ and add the edges $uv,yb$ and $ax$ to obtain a graph $G \in \Bs_{(\sbb,\tb)}(uv,M)$. Condition 5 ensures $u$ and $v$ are distinct from $a$ and $b$, which along with Condition 4 implies the edges $ux$, $yv$ and $ab$ are different. Conditions 6 and 7 ensure the edges $ax$ and $yb$ can be added to $G'$. Let $b_3^B(G')$ denote the number of backward switchings from $G'$. 

\begin{claim}\label{back_3_switch_bipartite_lem}
    For all $G' \in \Bs_{(\sbb,\tb)}(\emptygraph,M+uv)$,
    \begin{equation*}
         b_3^B(G') \leq m(G) s_u t_v
    \end{equation*}
    and
    \begin{equation*}
        b_3^B(G') \geq m(G)s_u t_v  \cdot q(\sbb,\tb,\emptygraph,M,uv).
    \end{equation*}
\end{claim}
\begin{proof}
    \textbf{Upper Bound.}
    For the upper bound, we count exactly the number of $(y,a,x,b)$ satisfying Condition 4. We can choose $x$ such that $ux \in G'$ in exactly $s_u$ ways. For each choice of $x$, there are exactly $t_v$ choices of $y$ with $yv \in G'$. For each choice of $x$ and $y$, we can choose $a$ and $b$ satisfying Condition 4 in exactly $m(G)$ ways. This gives the result.
    \textbf{Lower Bound.}
    For the lower bound we use inclusion-exclusion. We call a $4$-tuple \textit{terrible} if it satisfies Condition $4$ but not Condition 5. We call a $4$-tuple \textit{bad} if it satisfies Conditions 4 and 5 but not Condition 6 or 7. The lower bound on $b_3^B(G')$ arises by starting with the number of $4$-tuples satisfying Condition 4 and subtracting an exact count on the number of terrible $4$-tuples and an upper bound on the number of bad $4$-tuples.
     \par\textbf{Terrible Cases: Exact Count.} Please refer to Figure \ref{terrible_cases_bipartite}, which displays the terrible cases not satisfying Condition $5$. There are exactly $s_u t_v$ choices for $x$ and $y$. For each of these choices we are in case (A) or (B). For (A) we set $u = a$ and have $s_u$ choices for $b$, where we permit $b = x$. For (B) we set $v = b$ and have $t_v$ choices for $a$, where we permit $y = a$. Therefore, there are $s_u t_v (s_u + t_v)$ terrible $4$-tuples.
   \begin{figure}[htbp]
\centering
    \begin{tabularx}{0.95\textwidth}{*{2}{>{\centering\arraybackslash}X}}
    \begin{tikzpicture}[
every edge/.style = {draw=black,very thick},
 vrtx/.style args = {#1/#2}{%
      circle, draw, thick, fill=white,
      minimum size=4mm, label=#1:#2}
                    ]
\node(u) [vrtx= above/, label = above : {$u=a$}] at (-1.29904, 0.75) {};
\node(v) [vrtx= left/$v$,fill = black] at (-1.29904, -0.75) {};
\node(y) [vrtx= above/$y$] at (0, 0.75) {};
\node(x) [vrtx= below/$x$,fill = black] at (0, -0.75) {};
\node(b) [vrtx= right/$b$,fill = black] at (1.29904, -0.75) {};


\path   (u) edge[dashed] (v)
        (u) edge   (x)
        (v) edge  (y)
        (u) edge (b);
\end{tikzpicture} 
\centerline{(A) : $u = a$} 
& 
 \begin{tikzpicture}[
every edge/.style = {draw=black,very thick},
 vrtx/.style args = {#1/#2}{%
      circle, draw, thick, fill=white,
      minimum size=4mm, label=#1:#2}
                    ]
\node(u) [vrtx= left/$u$] at (-1.29904, 0.75) {};
\node(v) [vrtx= above/, label = below : {$v =b$},fill = black] at (-1.29904, -0.75) {};
\node(y) [vrtx= above/$y$] at (0, 0.75) {};
\node(x) [vrtx= above/, label = below : {$x$},fill = black] at (0, -0.75) {};
\node(a) [vrtx= right/$a$] at (1.29904, 0.75) {};

\path   (u) edge[dashed] (v)
        (u) edge  (x)
        (v) edge  (a)
        (v) edge  (y);
        \end{tikzpicture} 
    \centerline{(B) : $v = b$} 
    \end{tabularx}
\caption{Terrible Cases where Condition 4 holds but not Condition 5}
\label{terrible_cases_bipartite}
\end{figure}
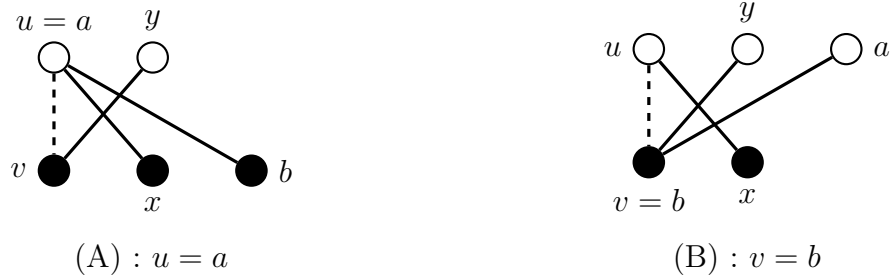 
    \par \textbf{Bad Cases: Upper Bound.}
    Let $(y,a,x,b)$ satisfy Conditions 4 and 5 but not Condition 6. We will upper bound the number of such $4$-tuples. Either (A) $x = b$ or (B) $ax \in G' \cup M$. Both are illustrated in Figure \ref{bad_con_6_bipartite}. Recall $(\mb,\nb)$ is the degree sequence of $M$.

    \begin{figure}[htbp]
\centering
\captionsetup{justification=centering} 
    \begin{tabularx}{0.95\textwidth}{*{2}{>{\centering\arraybackslash}X}}
 \begin{tikzpicture}[
every edge/.style = {draw=black,very thick},
 vrtx/.style args = {#1/#2}{%
      circle, draw, thick, fill=white,
      minimum size=4mm, label=#1:#2}
                    ]
\node(u) [vrtx= left/$u$] at (-1.29904, 0.75) {};
\node(v) [vrtx= left/$v$,fill = black] at (-1.29904, -0.75) {};
\node(y) [vrtx= above/$y$] at (0, 0.75) {};
\node(x) [vrtx= above/, label = below : {$x=b$},fill = black] at (0, -0.75) {};
\node(a) [vrtx= right/$a$] at (1.29904, 0.75) {};

\path   (u) edge[dashed] (v)
        (u) edge  (x)
        (x) edge  (a)
        (v) edge  (y);
        \end{tikzpicture} 
    \centerline{(A) : $x = b$} 

& 
\begin{tikzpicture}[
every edge/.style = {draw=black,very thick},
 vrtx/.style args = {#1/#2}{%
      circle, draw, thick, fill=white,
      minimum size=4mm, label=#1:#2}
                    ]
\node(u) [vrtx= left/$u$] at (-1.29904, 0.75) {};
\node(v) [vrtx= left/$v$,fill = black] at (-1.29904, -0.75) {};
\node(y) [vrtx= above/$y$] at (0, 0.75) {};
\node(x) [vrtx= below/$x$,fill = black] at (0, -0.75) {};
\node(a) [vrtx= right/$a$] at (1.29904, 0.75) {};
\node(b) [vrtx= right/$b$,fill = black] at (1.29904, -0.75) {};


\path   (u) edge[dashed] (v)
        (u) edge   (x)
        (x) edge[red] (a)
        (v) edge  (y)
        (a) edge (b);
\end{tikzpicture} 
\centerline{(B) : $ax \in G' \cup M$}
    \end{tabularx}
\caption{Bad cases where Conditions 4 and 5 hold but not Condition~6}
\label{bad_con_6_bipartite}
\end{figure}
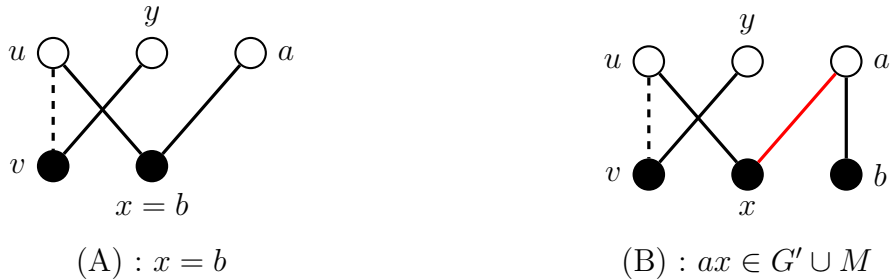

    We bound cases (A) and (B) simultaneously. There are $t_v$ choices for $y$. For each fixed~$y$, the number of $(y,a,x,b)$ satisfying Conditions 4 and 5 but not satisfying Condition 6 is given by
    \begin{equation*}
        \sum_{\substack{x \in T \\ ux \in G'}}  \sum_{\substack{a \in S \\ ax \in G' \cup M \\ a \neq u}} s_a .
    \end{equation*}
    This count arises as follows. Fix $x \in T$ such that $ux \in G'$. We count the number of directed 2 paths $x \to a \to b$ such that $ax \in G' \cup M$ and $ab \in G'$, where $a \neq u$. We allow $x = b$ when counting these directed 2-paths, so cases (A) and (B) are counted simultaneously. After fixing $x$, the number of such 2-paths is
    $$\sum_{\substack{a \in S \\ ax \in G' \cup M \\ a \neq u}} s_a,$$
    since for each choice of $a$ with $ax \in G'\cup M$ and $a \neq x$, there are $s_a$ choices for $b$. There are $t_x + n_x - 1$ terms in the above summation, and at most they consist of the $t_x + n_x -1$ largest degrees in $G'$, where none of these degrees are the term $s_u$. Therefore, 
    \begin{equation*}
        \sum_{\substack{x \in T \\ ux \in G'}}  \sum_{\substack{a \in S \\ ax \in G' \cup M \\ a \neq u}} s_a \leq \sum_{\substack{x \in T \\ ux \in G'}} S^{\{u\}}(t_x+n_x - 1) .
    \end{equation*}
    Since this is for a fixed $y$, the number of $4$-tuples satisfying Conditions 4 and 5 but not Condition 6 is
    \begin{equation*}
        t_v \sum_{\substack{x \in T \\ ux \in G'}} S^{\{u\}}(t_x+n_x - 1).
    \end{equation*}
    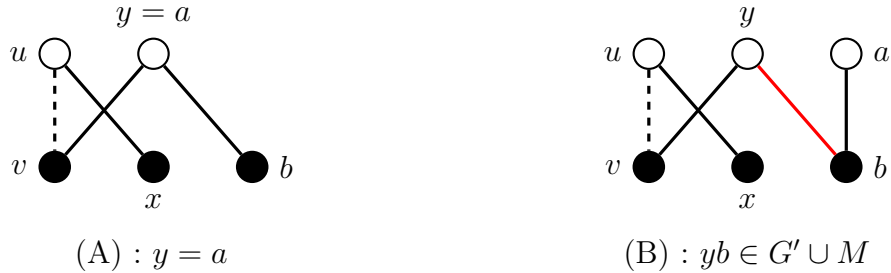
\begin{figure}[htbp]
\centering
    \begin{tabularx}{0.95\textwidth}{*{2}{>{\centering\arraybackslash}X}}
 \begin{tikzpicture}[
every edge/.style = {draw=black,very thick},
 vrtx/.style args = {#1/#2}{%
      circle, draw, thick, fill=white,
      minimum size=4mm, label=#1:#2}
                    ]
\node(u) [vrtx= left/$u$] at (-1.29904, 0.75) {};
\node(v) [vrtx= left/$v$,fill = black] at (-1.29904, -0.75) {};
\node(y) [vrtx= above/, label = above : {$y=a$}] at (0, 0.75) {};
\node(x) [vrtx= above/, label = below : {$x$},fill = black] at (0, -0.75) {};
\node(b) [vrtx= right/$b$,fill = black] at (1.29904, -0.75) {};

\path   (u) edge[dashed] (v)
        (u) edge  (x)
        (y) edge  (b)
        (v) edge  (y);
        \end{tikzpicture} 
    \centerline{(A) : $y = a$}  
& 
\begin{tikzpicture}[
every edge/.style = {draw=black,very thick},
 vrtx/.style args = {#1/#2}{%
      circle, draw, thick, fill=white,
      minimum size=4mm, label=#1:#2}
                    ]
\node(u) [vrtx= left/$u$] at (-1.29904, 0.75) {};
\node(v) [vrtx= left/$v$,fill = black] at (-1.29904, -0.75) {};
\node(y) [vrtx= above/$y$] at (0, 0.75) {};
\node(x) [vrtx= below/$x$,fill = black] at (0, -0.75) {};
\node(a) [vrtx= right/$a$] at (1.29904, 0.75) {};
\node(b) [vrtx= right/$b$,fill = black] at (1.29904, -0.75) {};


\path   (u) edge[dashed] (v)
        (u) edge   (x)
        (y) edge[red] (b)
        (v) edge  (y)
        (a) edge (b);
\end{tikzpicture} 
\centerline{(B) : $yb \in G' \cup M$} 
    \end{tabularx}
\caption{Bad cases where Conditions 4 and 5 hold but not Condition~7}
\label{fig:condition7_bad_bipartite}
\end{figure}

    Now we upper bound the number of bad $4$-tuples satisfying Conditions 4 and 5 but not satisfying Condition 7. Figure \ref{fig:condition7_bad_bipartite} displays these cases. By a  similar count, the number of these bad $4$-tuples is at most 
    \begin{equation*}
        s_u \sum_{\substack{y \in S \\ yv \in G'}} T^{\{v\}}(s_y+m_y - 1).
    \end{equation*}
    Therefore, the total number of bad and terrible $4$-tuples is at most
\begin{equation}\label{bipartite:bad_and_terrible}
     s_u^2 t_v +t_v \sum_{\substack{x \in T \\ ux \in G'}} S^{\{u\}}(t_x+n_x - 1)
     + t_v^2 s_u + s_u \sum_{\substack{y \in S \\ yv \in G'}} T^{\{v\}}(s_y+m_y - 1).
\end{equation}

Let us consider the first two terms (the reasoning to simplify the last two terms is similar). The first summation has $s_u$ terms, then we have
\begin{align*}
    s_u^2 t_v +t_v \sum_{\substack{x \in T \\ ux \in G'}} S_H^{\{u\}}(t_x+n_x - 1) & = t_v \Biggl( \sum_{\substack{x \in T \\ ux \in G'}} s_u +  S^{\{u\}}(t_x+n_x - 1)\Biggr) \\
    & \leq t_v \Biggl( \sum_{\substack{x \in T \\ ux \in G'}}  S(t_x+n_x )\Biggr) & \text{Lemma \ref{deg_sum_prop} (i)}\\
    & \leq t_v s_u S(\alpha_{u}(\tb,\nb,\sbb,\zerob)).
\end{align*}
The last line follows by Jensen's inequality and Lemma \ref{deg_sum_prop} (ii). Then \eqref{bipartite:bad_and_terrible} is at most $s_u t_v \left(S(\alpha_{u}(\tb,\nb,\sbb,\zerob)) + T(\alpha_{v}(\sbb,\mb,\tb,\zerob))\right) $.
Therefore,
\begin{equation*}
    b_3^B(G') \geq s_u t_v m(G)  -s_u t_v \left(S(\alpha_{u}(\tb,\nb,\sbb,\zerob)) + T(\alpha_{v}(\sbb,\mb,\tb,\zerob))\right),
\end{equation*}
which factors to the result.
\end{proof}

Let $T'$ denote the number of pairs $(G,G')$, where $G \in \Bs_{(\sbb,\tb)}(uv,M)$ and $G' \in \Bs_{(\sbb,\tb)}(\emptygraph,M+uv)$ and there exists a backward 3-switching from $G'$ to $G$. Observe that this implies there is also a forward 3-switching from $G$ to $G'$ and we have the equality
\begin{equation*}
 \sum_{G' \in \Bs_{(\sbb,\tb)}(\emptygraph,M+uv)} b_3^B(G') = T' = \sum_{G \in \Bs_{(\sbb,\tb)}(uv,M)} f_3^B(G) . 
\end{equation*}
Claim \ref{back_3_switch_bipartite_lem} and Claim \ref{forward_3_switch_upper_bipartite} then give
\begin{equation*}
    |\Bs_{(\sbb,\tb)}(\emptygraph,M + uv)| \cdot m(G) s_u t_v \cdot q(\sbb,\tb,\emptygraph,M,uv)  \leq |\Bs_{(\sbb,\tb)}(uv,M)| \cdot m(G)^2 . 
\end{equation*}
Observe that the substitutions $(\sbb,\tb) \mapsto (\sbb,\tb) - (\hb,\ib)$ and $M \mapsto H \cup L$ give the result as shown in Lemma \ref{ratio_bound_lemma_bipartite}. Indeed under the substitutions, $m(G) \mapsto m(G) - m(H)$ and $q(\sbb,\tb,\emptygraph,M,uv) \mapsto q(\sbb,\tb,H,L,uv)$.

\subsection{Proof of Corollaries \ref{multiple_edge_cor_bipartite} and \ref{multiple_forbidden_cor_bipartite}}
We will sketch the proofs in the bipartite case. They are almost identical to the generic case, however one keeps track of the functions $x \mapsto S(x)$ and $x \mapsto T(x)$ rather than the function $x \mapsto D(x)$.
\begin{proof}[Proof of Corollary \ref{multiple_edge_cor_bipartite}]
Consider the setup of Theorem \ref{multiple_edges_bipartite}. For each $1 \leq i \leq m(X)$, choose $u_i \in S$ and $v_i \in T$ such that  $e_i = u_iv_i$. We now bound the product of the functions $p$ and $q$ in Theorem \ref{multiple_edges_bipartite}.
    
    \textbf{Part A: Upper Bound.} For $k$ large using Lemma \ref{deg_sum_prop} (ii), we have
    \begin{align*}
    & \frac{S_{X_{i-1}}(t_{v_i} + m_{v_i}) + T_{X_{i-1}}(s_{u_i} + l_{u_i})}{m(G) - m(X_{i-1})} \\
    & \hspace{1cm} \leq \frac{S(\varDelta_{\partial{X}}(\tb + \mb)) + T(\varDelta_{\partial{X}}(\sbb + \lb))}{m(G) - m(X) } \leq \rho(k).    
\end{align*}
Theorem \ref{multiple_edges_bipartite}(A) then applies. Letting $C(k) \coloneq \frac{\ln(1-\rho(k))}{\rho(k)}$ and using $1-x \geq e^{C(k)x}$ for $x \in [0,\rho(k)]$ gives
\begin{align*}
    \prod_{i=1}^{m(X)}  p(\sbb,\tb, X_{i-1}, L, e_i) & \geq \exp\biggl( C(k) \cdot \sum_{i=1}^{m(X)} \frac{S_{X_{i-1}}(t_{v_i} + m_{v_i}) + T_{X_{i-1}}(s_{u_i} + l_{u_i}) }{m(G) - m(X_{i-1})}\biggr) \\
    & \geq \exp \cdot \left(C(k) \cdot \frac{m(X)(S(\kappa_T)+T(\kappa_S))}{m(G)-m(X)}\right) \\
    & \geq \exp \cdot \left(C(k) \cdot \frac{m(X)(S(\varDelta_{\partial{X}}(\tb + \mb))+T(\varDelta_{\partial{X}}(\sbb + \lb))) }{m(G)-m(X)}\right), 
\end{align*}
where $\kappa_S \coloneq \frac{\sum_{i\in S} x_{i}(s_i + l_i)}{m(X)}  $ and  $\kappa_T \coloneq  \frac{\sum_{i \in T} y_{i}(t_i + m_{i})}{m(X)}$. The second last equality follows by Lemma \ref{deg_sum_prop} (iii), observing that $\sum_{i=1}^{m(X)}S(t_{v_i} + m_{v_i}) = \sum_{i \in T} y_{i} S(t_i+m_{i})$, $\sum_{i=1}^{m(X)} T(s_{u_i} + l_{u_i}) = \sum_{i \in S} x_{i} T(s_i+l_{i})$, and then applying Jensen's inequality. The result now follows since $\rho(k)$ is bounded away from $1$.

\textbf{Part B: Lower Bound.}
Under the assumptions the result of Theorem \ref{multiple_edges_bipartite}(B) applies. Letting $C(k) \coloneq \frac{\ln(1-\rho(k))}{\rho(k)}$ and using $1-x \geq e^{C(k)x}$ for $x \in [0,\rho(k)]$ gives
\begin{align*}
    \prod_{i=1}^{m(X)} &q(\sbb,\tb, X_{i-1}, L, e_i) \\[-2ex]
    &{\quad}\geq \exp\biggl(C(k) \cdot \sum_{i=1}^{m(X)}  \frac{S_{X_{i-1}}(\alpha_{u_i}(\tb,\mb,\sbb,\xb_{i-1})) + T_{X_{i-1}}(\alpha_{v_i}(\sbb,\lb,\tb,\yb_{i-1}))}{m(G) - m(X_{i-1})} \biggr) \\
    & {\quad}\geq \exp\left(C(k) \cdot \frac{m(X)\left(S(\mu_S)  + T(\mu_T) \right)}{m(G) - m(X)}\right) \\
    & {\quad}\geq \exp\left(C(k) \cdot \frac{m(X)\left(S(\gamma(\tb,\mb,\sbb,\yb))  + T(\gamma(\sbb,\lb,\tb,\xb)) \right)}{m(G) - m(X)}\right),
\end{align*}
where $\mu_S \coloneq  \frac{\sum_{i \in S} x_{i} \alpha_{i}(\tb,\mb,\sbb,\xb)}{m(X)}  $ and $\mu_T \coloneq  \frac{\sum_{i \in T} y_{i} \alpha_{i}(\sbb,\lb,\tb,\yb)}{m(X)}$. The second last inequality follows by Lemma \ref{deg_sum_prop} (iii), observing that $\sum_{i=1}^{m(X)} S(\alpha_{u_i}(\tb,\mb,\sbb,\xb)) = \sum_{i \in S} x_{i} S(\alpha_{i}(\tb,\mb,\sbb,\xb)) $ (a similar equality holds for $T$), and then applying Jensen's inequality. The result now follows as $\rho(k)$ is bounded away from $1$ and $\phi'(\db,X) = o(1)$ or $O(1)$.
\end{proof}

\begin{proof}[Proof of Corollary \ref{multiple_forbidden_cor_bipartite}]
\textbf{Part A: Upper Bound.}
    The assumption implies the result of Theorem \ref{multiple_forbidden_theorem_bipartite}(A) holds. Letting $C(k) \coloneq \frac{\ln(1-\rho(k))}{\rho(k)}$ and using $1-x \geq e^{C(k)x}$ for $x \in [0,\rho(k)]$ gives
    \begin{align*}
       \prod_{j=1}^{m(Y)} q(\sbb,\tb,\emptygraph,L_{j-1},\ebar_j) & \geq \exp\left(C(k) \frac{m(Y)(S(\lambda_S)+T(\lambda_T))}{m(G)}\right) \\
       & \geq \exp\left(C(k) \frac{m(Y) (S(\gamma(\tb,\mb+\zb,\sbb,\zerob))+T(\gamma(\sbb,\lb + \yb, \tb, \zerob)))}{m(G)}\right),
    \end{align*}
where $\lambda_S \coloneq \frac{\sum_{i \in S} y_i \alpha_i(\tb,\mb + \zb, \sbb, \zerob)}{m(Y)}$ and $\lambda_T \coloneq \frac{\sum_{i \in T} z_i \alpha_i(\sbb,\lb + \yb, \tb, \zerob)}{m(Y)}$. The result now follows from the assumptions of the corollary.

\textbf{Part B: Lower Bound.}
 The assumption implies the result of Theorem \ref{multiple_forbidden_theorem_bipartite}(B) holds. Letting $C(k) \coloneq \frac{\ln(1-\rho(k))}{\rho(k)}$ and using $1-x \geq e^{C(k)x}$ for $x \in [0,\rho(k)]$ gives
    \begin{align*}
       \prod_{j=1}^{m(Y)} p(\sbb,\tb,\emptygraph,L_{j-1},\ebar_j) & \geq \exp\left(C(k) \frac{m(Y)(S(\eta_S)+T(\eta_T))}{m(G)}\right) \\
       & \geq \exp\left(C(k) \cdot \frac{m(Y) (S(\varDelta_{\partial Y}(\tb + \mb + \zb))+T(\varDelta_{\partial Y}(\sbb + \lb + \yb)))}{m(G)}\right),
    \end{align*}
where $\eta_S \coloneq \frac{\sum_{i \in S} y_i(s_i + l_i + y_i)}{m(Y)}$ and $\eta_T \coloneq \frac{\sum_{i \in T} z_i(t_i + m_i +z_i)}{m(Y)}$.
\end{proof}

\section{Sample Calculations}\label{section:example_calculation}
\textbf{Example 1: Degree sequence with linear maximum degree}.  
Consider the $n$-tuple $\db_n = (\lfloor \frac{n}{10} \rfloor, \lfloor \sqrt{n} \rfloor, \lfloor \sqrt{n} \rfloor, \lfloor \sqrt{n} \rfloor, \lfloor \ln n \rfloor, \lfloor \ln n \rfloor, \dots , \lfloor \ln n \rfloor)$ and $G \sim \Gs_{\db_n}(\emptygraph,\emptygraph)$. Let $X$ be a triangle between the three vertices of degree $\lfloor \sqrt{n} \rfloor$. We find $\Pa(X \subseteq G)$ using Corollary \ref{multiple_edge_cor}. Note that $D(\varDelta_{\partial X}(\db) ) = D(\lfloor \sqrt{n} \rfloor) = O(n) = o(m(G)-m(X))$ as $m(G) = \Theta(n\ln n)$. We then obtain $\Pa(X \subseteq G) \leq \frac{([ \sqrt{n} ]_2)^3}{2^3[m(G)]_3}(1+o(1))$. For the lower bound, the minimum degree in $\db-\xb$ is $\lfloor \ln n \rfloor$. Therefore, $\gamma(\db,\zerob,\db,\xb) = \frac{D(\db, \lfloor \ln n \rfloor)}{\lfloor \ln n \rfloor} = O(\frac{n}{\ln n})$. Then $D(\gamma(\db,\zerob,\db,\xb) ) = O(n) = o(m(G) - m(X))$. Further, $\varLambda(n) = o(1)$ and we obtain $\Pa(X\subseteq G) \geq \frac{([ \sqrt{n} ]_2)^3}{2^3[m(G)]_3} (1+o(1))$. This yields $\Pa(X\subseteq G) = \frac{1}{ \\ln^3 n}(1+o(1))$. Notice that the same result can not be obtained using \cite[Corollary 5]{gao22} as $D(\varDelta(\db)) = \Theta(m(G))$.

\textbf{Example 2: Example 1 Conditioned on Forbidden Edges}. Let $L$ be a spanning $r$-regular subgraph with $r = o(n)$. One can repeat the analysis in Example 1, where instead $G \sim \Gs_{\db_n}(\emptygraph,L)$ and $X \cap L = \emptygraph$. This results in the same asymptotic probability for both upper and lower bounds despite conditioning on all edges of $L$ not being present. To obtain the same result from Theorem 4 in \cite{gao22} one would need $m(L)\cdot \varDelta(\db)^2 = \Theta(r n^3) = o(m(G)^2)$ for the upper bound. However, even for $r = 1$ this does not hold.

\textbf{Example 3: Cycle of length $\lfloor \sqrt{n} \rfloor$}.
The proofs of Corollaries \ref{multiple_edge_cor} and \ref{multiple_forbidden_cor} and their bipartite counterparts are essentially a simplification of their corresponding theorems. In some cases one should stop during the simplification to obtain a more accurate result than the corollary as we now show.

Let $\db_n = (\lfloor \sqrt{n} \rfloor, \lfloor \ln{n} \rfloor, \lfloor \ln{n} \rfloor, \dots, \lfloor \ln{n} \rfloor)$ be an $n$-tuple and $X \subseteq K_n$ a cycle between the vertices in $\{1,2,\dots,\lfloor \sqrt{n} \rfloor\}$. We find an asymptotic upper bound on $\Pa(X \subseteq G)$ for $G \sim \Gs_{\db_n}(\emptygraph,\emptygraph)$. We have $D(\varDelta_{\partial X}(\db)) = D(\lfloor \sqrt{n} \rfloor) = O(\sqrt{n}\ln n) = o(m(G) - m(X))$ and the assumptions of Corollary \ref{multiple_edge_cor} are satisfied. Then one can take $\rho(n) \coloneq \frac{D(\varDelta_{\partial X}(\db)) }{m(G)-m(X)}$. However, $m(X) \cdot \rho(n) = \Omega(1)$, which implies that we can not apply the corollary to obtain the upper bound to a $1+o(1)$ factor. However, we verify as in the proof of Corollary \ref{multiple_edge_cor} the assumptions of Theorem \ref{multiple_edge_thm}(A) are satisfied. Then we can simplify the error term there into \eqref{kappa_expression}. One can verify that $\kappa \coloneq \frac{\sum_{i=j}^n x_j d_j}{2m(X)} = O(\ln n)$, which is  much less than $\varDelta_{\partial X}(\db)$. Since $m(X) D(\kappa) = O(n) = o(m(G) - m(X))$, we have that \eqref{kappa_expression} is in $1+o(1)$ as required. The asymptotic probability is then bounded by $\Pa(X \subseteq G) \leq \frac{n (\ln n)^{2 m(X) - 2} }{2^{m(X)}[m(G)]_{m(X)}}(1+o(1))$.

\textbf{Example 4: Forbidden Edges}.
Let $\db_n = (\lfloor  \frac{n}{4} \rfloor, \lfloor \sqrt{n} \rfloor, \lfloor \ln n \rfloor, \dots, \lfloor \ln n \rfloor)$ be an $n$-tuple and $G \sim \Gs_{\db_n}(\emptygraph,\emptygraph)$. Let $Y$ be a triangle between the vertices in $\{1,2,3\}$, the edges of which we will forbid. Note that $\gamma(\db,\yb,\db,\zerob) = O(\frac{n}{\ln n})$ and $D(\gamma(\db,\yb,\db,\zerob) +2) = O(n) = o(m(G))$. Corollary \ref{multiple_forbidden_cor} then gives $\Pa(Y \cap G = \emptygraph) \leq \frac{4}{5} \left(1 + \frac{\sqrt{n}}{4\ln n} \right)^{-1}(1+o(1))$. The $\frac{4}{5}$ factor comes from the edge $\{1,3\}$, and the $\left(1 + \frac{\sqrt{n}}{4\ln n} \right)^{-1}$ factor is from the edge $\{1,2\}$. Interestingly, there is no significant contribution to the asymptotic probability from the edge $\{2,3\}$ since with high probability this edge $\{2,3\}$ is not present. Further, $D(\varDelta_{\partial Y}(\db + \yb)) \leq \rho(n) \cdot m(G)$ for $\rho(n) = \frac{2}{3}$ and sufficiently large $n$. We have $m(Y) \cdot \rho(k) = O(1)$ and  $\Pa(Y \cap G = \emptygraph) \geq \frac{4}{5} \left(1 + \frac{\sqrt{n}}{4\ln n} \right)^{-1} \cdot \Omega(1)$.

\textbf{Example 5: Bipartite Case}.
Let $\sbb = (\lfloor \frac{n}{2} \rfloor, \lfloor \sqrt{n} \rfloor, \lfloor \sqrt{n} \rfloor, \dots, \lfloor \sqrt{n} \rfloor)$ be an $n$-tuple and $\tb = (\lfloor \ln n \rfloor, \lfloor \ln n \rfloor, \dots, \lfloor \ln n \rfloor)$ be a $\frac{n^{3/2}}{\ln n}(1+o(1))$-tuple. Let $G \sim \Bs_{(\sbb,\tb)}(\emptygraph,\emptygraph)$ be a random bipartite graph and $X$ a $4$-cycle between any four vertices. We have $S(\varDelta(\tb)) + T(\varDelta(\sbb)) = O(\varDelta(\sbb)\varDelta(\tb)) = O(n\ln n) = o(m(G) - m(X))$, then both the upper and lower bounds of Corollary \ref{multiple_edge_cor_bipartite} are satisfied. If $X$ contains the vertex of the largest degree in $\sbb$, then Corollary \ref{multiple_edge_cor_bipartite} yields $\Pa(X \subseteq G) = \frac{\ln^4 n }{4 n^3} \cdot (1 + o(1))$, where we used $m(G) = n^{3/2}(1+o(1))$. Otherwise, $\Pa(X \subseteq G) = \left(\frac{\ln n}{n}\right)^4(1+o(1))$. Notice that the same result can not be obtained from \cite[Theorem 3.5]{mckay81} as the condition $(\max\{\varDelta(\sbb),\varDelta(\tb)\})^2 = o(m(G) - m(X))$ is not satisfied.

\bibliographystyle{plain}

\end{document}